\documentclass[12pt]{amsart} 
\usepackage{amsmath, amsthm, mathtools, calrsfs, verbatim, bbm, color, graphics, geometry,graphicx,algorithm2e,accents,enumitem,float,subfigure}
\usepackage{todonotes,bm}
\usepackage[f]{esvect}
\usepackage[compress,noadjust]{cite}
\usepackage{amssymb,wasysym}

\newcommand{\bR}{\mathbb{R}}

\newcommand{\p}{\partial}
\newcommand{\epsi}{\varepsilon}
\newcommand{\ubar}[1]{\underaccent{\bar}{#1}}

\newcommand\restr[2]{{
  \left.\kern-\nulldelimiterspace 
  #1 
  \littletaller 
  \right|_{#2} 
  }}

\newcommand{\mgn}{\mu_{\Gamma}^-}

\newcommand{\av}{A[v]}

\newcommand{\bl}{\lambda^*}
\newcommand{\hrn}{\bR^n_+}

\theoremstyle{plain}
\newtheorem{theorem}{Theorem}[section]
\newtheorem{lemma}{Lemma}[section]
\newtheorem{corollary}{Corollary}[section]
\newtheorem{proposition}{Proposition}[section]

\theoremstyle{definition}

\newtheorem{question}{Question}[section]

\theoremstyle{remark}
\newtheorem{remark}{Remark}[section]
\newtheorem{example}{Example}[section]

\makeatletter
\def\thmhead@plain#1#2#3{%
  \thmname{#1}\thmnumber{\@ifnotempty{#1}{ }\@upn{#2}}%
  \thmnote{ {\the\thm@notefont#3}}}
\let\thmhead\thmhead@plain
\makeatother

\begin{document}

\title[Fully nonlinear Yamabe problem on manifolds with boundary]{On the fully nonlinear Yamabe problem with constant boundary mean curvature. I
}

\author[B.Z. Chu]{BaoZhi Chu}
\address[B.Z. Chu]{Department of Mathematics, Rutgers University, 110 Frelinghuysen Road, Piscataway, NJ 08854-8019, USA}
\email{bc698@math.rutgers.edu}

\author[Y.Y. Li]{YanYan Li}
\address[Y.Y. Li]{Department of Mathematics, Rutgers University, 110 Frelinghuysen Road, Piscataway, NJ 08854-8019, USA}
\email{yyli@math.rutgers.edu}

\author[Z. Li]{Zongyuan Li}
\address[Z. Li]{Department of Mathematics, City University of Hong Kong, 83 Tat Chee Avenue, Kowloon Tong, Hong Kong SAR}
\email{zongyuan.li@cityu.edu.hk}

\begin{abstract} 
In a recent paper, we established optimal Liouville-type theorems for conformally invariant second-order elliptic equations in the Euclidean space. In this work, we prove an optimal Liouville-type theorem for these equations in the half-Euclidean space.
\end{abstract}
\maketitle
	
\vspace{.25in}

\section{Introduction}
The Yamabe problem was solved through the works of Yamabe, Trudinger, Aubin, and Schoen: On a compact smooth connected Riemannian manifold without boundary, there exist constant scalar curvature conformal metrics. Its fully nonlinear version has been studied extensively for over twenty years; see \cite{CGY-annals, CGY-janalmath, Ge-Wang, guan-wang-06,gursky-viaclovsky-indiana,gursky-viaclovsky-annals, Li-Li_CPAM03, LiLi_acta, Li-Nguyen-14, Li-Nguyen-greensfunc, sheng-trudinger-wang, viaclovsky-duke00,viaclovsky-02} and the references therein. 

In \cite{CLL-1}, we have broadened the scope of the fully nonlinear Yamabe problem by
establishing optimal Liouville-type theorems, local gradient estimates, and new existence and compactness results, allowing the scalar curvature of the conformal metrics to have varying signs. 
In this paper, we establish an optimal Liouville-type theorem associated with the fully nonlinear Yamabe problem with constant boundary mean curvature.

On an $n$-dimensional Riemannian manifold $(M^n,g)$ with boundary $\p M$, $n\geq 3$,
consider the Schouten tensor
\begin{equation*}
    A_g=\frac{1}{n-2}\big( Ric_g-\frac{R_g}{2(n-1)}g\big),
\end{equation*}
where $Ric_g$ and $R_g$ denote, respectively, the Ricci tensor and the scalar curvature.
We denote by $\lambda(A_g)=(\lambda_1(A_g),\dots,\lambda_n(A_g))$ the eigenvalues of $A_g$ with respect to $g$. 
We use $h_g$ to represent the mean curvature on $\p M$ with respect to the unit inner normal $\vec{n}_g$ on $\p M$ (the boundary of a Euclidean ball has positive mean curvature).

Let, as in \cite{CLL-1}, 
 \begin{equation} \label{eqn-230331-0110}
   \begin{cases}
    \Gamma \subsetneqq \bR^n \,\, \text{be a non-empty open symmetric cone\footnotemark with vertex at the origin},\\
    \Gamma+\Gamma_n\subset\Gamma,
    \end{cases}
\end{equation}\footnotetext{By symmetric set, we mean that $\Gamma$ is invariant under interchange of any two $\lambda_i$.}where 
$\Gamma_n:=\{\lambda\in \bR^n\ |\ \lambda_i>0,~\forall 1\leq i\leq n\}$. In condition \eqref{eqn-230331-0110}, neither $\Gamma\subset\Gamma_1$ nor the convexity of $\Gamma$ is assumed, where $\Gamma_1\coloneqq\{ \lambda\in \bR^n \mid \sum_{i=1}^n \lambda_i>0\}$.
On the other hand, it is easy to see that \eqref{eqn-230331-0110} implies $\Gamma_n\subset \Gamma\subset \bR^n\setminus (-\overline{\Gamma_n})$. 

\begin{question}\label{quest-1.1-schouten-meancurvature}
Let $(M^n, g)$ be a compact smooth Riemannian manifold with boundary, $n\geq 3$, and let $\Gamma$ satisfy \eqref{eqn-230331-0110}. Assume that $\lambda(A_g)\in\Gamma$ on $M^n$. 
Given a constant $c$,
for which symmetric function\footnote{By symmetric function, we mean that $f$ is invariant under permutation of $\lambda_i$'s.} $f$ defined on $\Gamma$, does there exist a function $v$ on $M^n$ such that the conformal metric $g_v\coloneqq e^{2v}g$ satisfies
\begin{equation*}
         {f}(\lambda(A_{g_v}))=1\quad \text{on}~M^n,
\end{equation*}
and the boundary mean curvature $h_{g_v}$ satisfies 
\begin{equation*}
         h_{g_v}=c\quad \text{on}~\p M\text{?}
\end{equation*}
\end{question}

For $(f,\Gamma)=(\sigma_1,\Gamma_1)$, the above question is the boundary Yamabe problem.
When $\Gamma\subset\Gamma_1$ and $f$ is a concave function in $\Gamma$, Question \ref{quest-1.1-schouten-meancurvature} was proposed in \cite{Li-Li_JEMS}; see \cite{sophie-cvpde, chen-wei-100, He-Sheng, Jiang-Trudinger-oblique-1,Jiang-Trudiunger-2021, Jin-07, Jin-Li-Li, Li-Li_JEMS, Li2009, Li-Nguyen-umblic, Li-Nguyen-counterexample} and the references therein for studies on this problem.

On the other hand, when $\Gamma$ is not contained in $\Gamma_1$, allowing the scalar curvature $R_{g_v}$ to have varying signs, there has been no study of Question \ref{quest-1.1-schouten-meancurvature} in the literature.

\smallskip

In this paper, as a first step towards answering Question \ref{quest-1.1-schouten-meancurvature}, we establish a Liouville-type theorem on half Euclidean space
as mentioned earlier.

For a conformal metric $g_v\coloneqq e^{2v} g$, it is known that 
\begin{equation*}
A_{g_v}= -\nabla^2 v +  d v\otimes d v -\frac{1}{2} |\nabla v|^2 g + A_g,\quad \text{and}\quad h_{g_v}=e^{-v}(h_g-\frac{\p v}{\p \vec{n}_g}),
\end{equation*}
where all covariant derivatives and norms on the above are with respect to $g$. In particular, when $\bar{g} = |dx|^2$ is the Euclidean metric on $\overline\hrn$, 
\begin{equation*}
A_{\bar{g}_{v} } =e^{2 v} \av_{ij} dx^i dx^j,\quad \text{and}\quad h_{\bar{g}_v}=-e^{-v}\frac{\p v}{\p x_n},
\end{equation*}
where we denote
$\hrn\coloneqq \{(x_1,\dots,x_n)\in\bR^n\mid x_n>0\}$, and
$\av$ is the M\"obius Hessian of $v$ defined as below.

\subsection{Main results}
In dimensions $n\geq 2$, let $v$ be a $C^2$ function on $\bR^n$. The M\"obius Hessian of $v$ is defined as the following $n\times n$ symmetric matrix: 
\begin{equation*} 
    \av = e^{-2v} \left( - \nabla^2 v + \nabla v \otimes \nabla v - \frac{1}{2} |\nabla v|^2 I \right),
\end{equation*}
where $I$ is the $n\times n$ identity matrix. 
It is known that $\av$ has the following M\"obius invariance:
\begin{equation*}
    \lambda(A[v^\varphi])=\lambda(\av)\circ \varphi,\quad v^{\varphi} := v \circ \varphi + \frac{1}{n}\log |J_\varphi|,
\end{equation*}
where
$\varphi:\bR^n\cup\{\infty\}\to\bR^n\cup\{\infty\}$ is a M\"obius transformation, i.e. a finite composition of translations, dilations and inversions, $J_\varphi$ is the Jacobian of $\varphi$, and $\lambda(M)$ denotes the eigenvalues of the symmetric matrix $M$ modulo permutations.

Let $(f,\Gamma)$ satisfy the following conditions:
\begin{equation} \label{eqn-240223-0305}
\begin{cases}
f \in C^{0,1}_{loc}(\Gamma)\,\,\text{is a symmetric function and satisfies} ~~
    \frac{\p f}{\p \lambda_i} \geq c(K)>0,~\forall i\\
    \text{a.e. on compact subset} ~K~\text{of}~\Gamma,
\end{cases}
\end{equation}
\begin{equation}\label{natural-assumption}
    0\notin\overline{f^{-1}(1)}.
\end{equation}
Clearly, a symmetric function $f\in C^1(\Gamma)$ with $\p_{\lambda_i} f>0$ for any $i$ satisfies \eqref{eqn-240223-0305}. If $f\in C^0(\overline\Gamma)$ and $f|_{\p\Gamma}=0$, then condition \eqref{natural-assumption} is satisfied.

\medskip

For constant $c\in\bR$, consider the equation
\begin{equation}\label{halfspace-equ-v-critical}
    \begin{cases}
        f(\lambda(\av))=1\quad \text{in}~\overline\hrn,\\
        \frac{\p v}{\p x_n}=c\cdot e^v\quad \text{on}~\p \hrn.
    \end{cases}
\end{equation}
Denote, as in \cite{CLL-1},
\begin{equation*}
    \bl\coloneqq (1,-1,\dots,-1).
\end{equation*}
We have the following Liouville-type theorem.

\begin{theorem}\label{nondegenerateliouville-critical}
    For $n\geq 2$ and $c\in\bR$, let $(f,\Gamma)$ satisfy \eqref{eqn-230331-0110}--\eqref{natural-assumption} and $\bl\notin \overline\Gamma$, and $v\in C^2(\overline\hrn)$ satisfy \eqref{halfspace-equ-v-critical}. Then \begin{equation}\label{half-space-bubble}
        v(x)\equiv \log \left(\frac{a}{1+b|x-\bar{x}|^2}\right),
    \end{equation}
    where $a,b>0$ and $\bar x=(\bar x', \bar x_n)\in\bR^n$ satisfy $f(2a^{-2}b\bm{e})=1$ and $2 a^{-1}b \bar x_n=c$ with $\bm{e}=(1,1,\dots,1)$.
\end{theorem}

Note that for $v$ defined in \eqref{half-space-bubble}, $\av\equiv 2a^{-2}bI$ and $\frac{\p v}{\p x_n}e^{-v}=2a^{-1}b\bar{x}_n$ on $\p\hrn$.

\begin{remark}\label{nondegoptimalremark}
    The condition $\bl\notin\overline\Gamma$ in Theorem \ref{nondegenerateliouville-critical} is optimal: Whenever $\Gamma$ satisfies \eqref{eqn-230331-0110} with $\bl\in\overline\Gamma$, there exists a smooth function $f$ satisfying \eqref{eqn-240223-0305} and \eqref{natural-assumption} such that for any $c\in \bR$, there exists a smooth solution $v$ of \eqref{halfspace-equ-v-critical} which is not of the form \eqref{half-space-bubble}. For details, see Section \ref{counterdetail}.
\end{remark}

The condition $\bl\notin\overline\Gamma$ is equivalent to stating that for some $\mu>1$, the inequality $\lambda_1+\mu\lambda_2>0$ holds for any $\lambda\in\Gamma$ satisfying $\lambda_1\geq \dots\geq \lambda_n$.

Condition \eqref{natural-assumption} in Theorem \ref{nondegenerateliouville-critical} cannot be simply removed when $c\neq 0$. For $n\geq 2$ and some constant $\mu>n-1$, let
\begin{equation*}
    f(\lambda)\coloneqq  \min_{1\leq i \leq n} \left(\lambda_i + \tfrac{\mu}{n-1}\textstyle\sum_{j\neq i}\lambda_j \right) + 1,~ \Gamma\coloneqq \{\lambda\in\bR^n\mid \min_{1\leq i\leq n}(\lambda_i+\tfrac{\mu+1}{2(n-1)}\textstyle\sum_{j\neq i}\lambda_j) >0\},
\end{equation*}
then $\bl\notin\overline\Gamma$, and $v(x)=2(\mu-1)^{-1}\log(1+2^{-1}(\mu-1)c  x_n)$ is a solution of \eqref{halfspace-equ-v-critical}.

When $c=0$, Theorem \ref{nondegenerateliouville-critical} still holds without condition \eqref{natural-assumption}: Extending $v$ to $\bR^n$ evenly in $x_n$, then $v\in C^2(\bR^n)$, $f(\lambda(\av))=1$ on $\bR^n$, and the desired conclusion follows from \cite[Theorem 1.1]{CLL-1}.

\smallskip

In dimensions $n\geq 3$, Theorem \ref{nondegenerateliouville-critical} in the case $\Gamma\subset\Gamma_1$ was previously known. When $(f,\Gamma)=(\sigma_1,\Gamma_1)$, it was proved by Li and Zhu \cite{Li-Zhu}; while under an additional hypothesis $\limsup_{|x|\to\infty}(v(x)+2\log|x|)<\infty$, it was an earlier result of Escobar \cite{escobar_cpam}. For general $(f,\Gamma)$ with $\Gamma\subset\Gamma_1$ and $f\in C^1$, it was proved by Li and Li \cite{Li-Li_JEMS}.

We have been informed recently by Duncan and Nguyen that they have proved in \cite{duncan-nguyen-email}  some  Liouville-type theorems for the fully nonlinear Loewner-Nirenberg problem in the half space.

Liouville-type theorems in $\bR^n$ for equation $f(\lambda(\av))=1$ corresponding to Theorem \ref{nondegenerateliouville-critical} have been established. The case when $n\geq 3$ and $(f,\Gamma)=(\sigma_1,\Gamma_1)$ was proved by Caffarelli, Gidas, and Spruck \cite{CGS}. The case when $n=4$ and $(f,\Gamma)=(\sigma_2,\Gamma_2)$ was proved by Chang, Gursky, and Yang \cite{CGY-janalmath}. The case when $n\geq 3$ and $(f,\Gamma)=(\sigma_k,\Gamma_k)$ for all $k$ was proved by Li and Li \cite{Li-Li_CPAM03}. Later in \cite{LiLi_acta}, they extended their result  to general $(f,\Gamma)$ with $\Gamma\subset\Gamma_1$. The case when $n=2$ was proved by Li, Lu, and Lu \cite{Li2021ALT}. Recently, we established in \cite{CLL-1} an optimal Liouville-type theorem for the equation in $\bR^n$.
For classification results under some additional assumptions on solutions, see
\cite{aubin-bubble, talenti, viaclovsky-duke00, GNN,obata, Chen-Li}.

In \cite{CLL-1}, we identified the optimal condition $\bl\notin\overline\Gamma$ for the validity of Liouville-type theorem in $\bR^n$ for equation $f(\lambda(\av))=1$.
For the validity of Liouville-type theorem for equation \eqref{halfspace-equ-v-critical} in the half Euclidean space,
 $\bl\notin\overline\Gamma$ is also a necessary and sufficient condition.
One crucial ingredient in the proof of Theorem \ref{nondegenerateliouville-critical} is to treat isolated boundary singularities. For instance, let $v$ satisfy \eqref{halfspace-equ-v-critical}. By the M\"obius invariance of \eqref{halfspace-equ-v-critical}, we have
\begin{equation*}
     \begin{cases}
        f(\lambda(A[v^{\varphi_{0,1}}]))=1\quad &\text{in}~\overline\hrn\setminus\{0\},\\
        \frac{\p }{\p x_n}v^{\varphi_{0,1}}=c\cdot e^{v^{\varphi_{0,1}}}\quad &\text{on}~\p \hrn\setminus\{0\},
    \end{cases}\end{equation*}
where $\varphi_{0,1}(x)=x/|x|^2$ is the inversion with respect to the unit sphere $\p B_1(0)$ in $\bR^n$. Assume that $\bl\notin\overline\Gamma$ and $v^{\varphi_{0,1}}>v$ in $B^+$, we prove that
\begin{equation*}
\liminf_{x\to 0}(v^{\varphi_{0,1}}-v)(x)>0.
\end{equation*}
See Proposition \ref{ms_starter} and \ref{notouching-thm-av-critical}, where more general results are obtained.

\subsection{Examples}
 In the following, we provide some examples of Theorem \ref{nondegenerateliouville-critical} by choosing appropriate $(f,\Gamma)$.  Some examples are expressed in terms of the Ricci tensor. They follow from Theorem \ref{nondegenerateliouville-critical} through a linear transformation, see Appendix \ref{riccisection} for details as well as equivalent reformulations of our theorems in terms of the Ricci tensor.
Further examples can be found by considering $(f,\Gamma)$ in \cite[\S 9]{CLL-1}.

For any symmetric subset $\Omega$ of $\bR^n$, a continuous function defined on $\{\lambda\in\Omega\mid\lambda_1\geq \dots\geq\lambda_n\}$ corresponds to a continuous symmetric function on $\Omega$. Therefore,
we will only specify the definition of a symmetric function in the region $\{\lambda_1\geq\dots\geq\lambda_n\}$. Recall that $\bar{g}=|dx|^2$ is the Euclidean metric on $\bR^n$ and $h_{\bar{g}_v}$ is the mean curvature with respect to the conformal metric $\bar{g}_v\coloneqq e^{2 v}\bar{g}$.

\begin{example}
    For $n\geq 3$, $2\leq i\leq n$, and $c\in\bR$,
    \begin{equation*}
        \lambda_i(Ric_{\bar{g}_v})=1,~v\in C^2~\text{on}~\overline{\hrn},~\text{and}~h_{\bar{g}_v}=c~\text{on}~\p\hrn\implies v~\text{is given by \eqref{half-space-bubble}},
    \end{equation*}
where  
 $a,b>0$ and $\bar x=(\bar x', \bar x_n)\in\bR^n$ satisfy $4(n-1)ba^{-2}=1$ and $2 a^{-1}b \bar x_n=-c$. However, this result fails when $i=1$ for any $c$, as shown by counterexamples given in Remark \ref{nondegoptimalremark}.
    
\end{example}

\begin{example}
 For $n\geq 3$, $1\leq i<j\leq n$, and $c\in\bR$,
\begin{equation*}
(\lambda_i+\dots+\lambda_j)(Ric_{\bar{g}_v})=1,~v\in C^2~\text{on}~\overline{\hrn},~\text{and}~h_{\bar{g}_v}=c~\text{on}~\p\hrn\implies v~\text{is given by \eqref{half-space-bubble}},
    \end{equation*}
where $a,b>0$ and $\bar x=(\bar x', \bar x_n)\in\bR^n$ satisfy $4(n-1)(j-i+1)ba^{-2}=1$ and $2 a^{-1}b \bar x_n=-c$.
\end{example}
    For $n\geq 3$ and $p=1,\dots,n-1$, let
\begin{equation*}
    G_p(\lambda)\coloneqq p \textstyle\sum\limits_{i\leq n-p}\lambda_i + (n-p)\sum\limits_{i>n-p}\lambda_i.
\end{equation*} 
On a locally conformally flat Riemannian manifold $(M^n,g)$, the quantity $G_p(\lambda(A_g))$, sometimes referred to as the $p$-Weitzenb\"ock curvatures, arises naturally from the 
Weitzenb\"ock formula for $p$-forms $\omega$:
\begin{equation*}
    \bigtriangleup \omega=\nabla^*\nabla \omega + G_p(\lambda(A_g)) \omega,
\end{equation*}
where $\bigtriangleup=dd^*+d^*d$ is the Hodge-de Rham Laplacian and $\nabla^*\nabla$ is the connection Laplacian; see \cite{MR2306044} and the references therein.
\begin{example}
       For $n\geq 3$, $1\leq p\leq n-2$, and $c\in\bR$,
    \begin{equation*}
        G_p(\lambda(A_{\bar{g}_v}))=1,~v\in C^2~\text{on}~\overline{\hrn},~\text{and}~h_{\bar{g}_v}=c~\text{on}~\p\hrn\implies v~\text{is given by \eqref{half-space-bubble}},
    \end{equation*}
where  
 $a,b>0$ and $\bar x=(\bar x', \bar x_n)\in\bR^n$ satisfy $2p(n-p)ba^{-2}=1$ and $2 a^{-1}b \bar x_n=-c$. However, this result fails when $p=n-1$ for any $c$, as shown by counterexamples given in Remark \ref{nondegoptimalremark}.\end{example}

\subsection{More general equations.}
For constants $c\in\bR$ and $p\geq 0$, consider the following more general equations, which include \eqref{halfspace-equ-v-critical} and its subcritical cases, 
\begin{equation}\label{halfspace-equ-v}
\begin{cases}
        f(\lambda(\av))=e^{-pv}\quad  \text{in}~\overline\hrn, \\
    \frac{\p v}{\p x_n}=c\cdot e^v\quad  \text{on}~\p \hrn.
    \end{cases}
\end{equation}
We have the following Liouville-type theorem.
\begin{theorem}\label{nondegenerateliouville}
    For $n\geq 2$, $p\geq 0$ and $c\in\bR$, let $(f,\Gamma)$ satisfy \eqref{eqn-230331-0110} and \eqref{eqn-240223-0305} with $\bl\notin \overline\Gamma$, and let $v\in C^2(\overline\hrn)$ satisfy \eqref{halfspace-equ-v}. Then one of the following holds:
\begin{enumerate}[label=(\roman*)]
    \item $p=0$ and v is of the form  \eqref{half-space-bubble},
    where $a,b>0$ and $\bar x=(\bar x', \bar x_n)\in\bR^n$ satisfy $f(2a^{-2}b\bm{e})=1$ and $2 a^{-1}b \bar x_n=c$ with $\bm{e}=(1,1,\dots,1)$.
     \item The solution $v$ depends only on $x_n$, i.e. $v(x',x_n)\equiv v(0',x_n)$.
      \end{enumerate}
\end{theorem}

Note that the above theorem does not assume condition \eqref{natural-assumption}.

When either $c\leq 0$ or $\bm{e_n}\in\p\Gamma$ holds, Alternative (ii) does not occur. See Remark \ref{supple-remark}. Here and throughout the paper, we denote $\bm{e_n}\coloneqq (0,\dots,0,1)$.
In particular, Alternative (ii) does not occur when $(f,\Gamma)=(\sigma_k,\Gamma_k)$ for $2\leq k\leq n$, since $\bm{e_n}\in \p\Gamma_k$ for $2\leq k\leq n$. Here, 
$\sigma_k(\lambda)\coloneqq \textstyle\sum_{1\leq i_1< \cdots < i_k\leq n} {\lambda_{i_1}\cdots \lambda_{i_k} }$ and $\Gamma_k \coloneqq \{\lambda\in\bR^n \mid \sigma_{l}(\lambda)>0,~\forall 1\leq l \leq k\}$.

On the other hand, when both $\bm{e_n}\notin\p\Gamma$, i.e. $\mgn<\infty$, and $c>0$ hold, Alternative (ii) does occur in some cases. Recall that, as in \cite{CLL-1}, we define
\begin{equation*}
      \mgn\coloneqq \inf \{c\mid  (c,-1,\dots,-1)\in \overline{\Gamma}\}\in [0,\infty]. 
\end{equation*}
Here we follow the convention: $\inf \emptyset = \infty$. See \cite[\S 2.2]{CLL-1} for properties of $\mgn$.
We further assume the following condition on $(f,\Gamma)$:
\begin{equation}\label{nice-ode-condition}
    f~\text{is homogeneous of degree}~1~\text{and}~\overline{f^{-1}(1)}\cap \p\Gamma=\emptyset.
\end{equation}
Then Alternative (ii) never occurs if and only if $p\in[0,\mgn+1]$.
See Proposition \ref{affine-ode-blowup} and Remark \ref{remark-optimal-range}.

\subsection{Notations}
For a set $\Omega\subset\bR^n$, we denote $\Omega^+\coloneqq \Omega\cap\hrn$, $\p'\Omega\coloneqq$ the interior of the set $\p \Omega^+\cap \p \bR^n_+$, and $\p''\Omega\coloneqq \p\Omega^+\setminus\p'\Omega$.

\subsection{Organization of the paper} In \S \ref{proof-1.1}, we prove Theorem \ref{nondegenerateliouville-critical}. 
In \S \ref{proof-1.2}, we prove 
Theorem \ref{nondegenerateliouville}.
In \S \ref{ode-sec}, we discuss solutions depending only on one variable $x_n$.

\section{Proof of Theorem \ref{nondegenerateliouville-critical}}\label{proof-1.1}
The proof of Theorem \ref{nondegenerateliouville-critical} uses the by-now-standard method of moving spheres, a variant of method of moving planes. The crucial new ingredients of the proof are the following two propositions.
\begin{proposition}\label{ms_starter} 
For $n\geq 2$ and $c\in\bR$, let $\Gamma$ satisfy \eqref{eqn-230331-0110} with $\bl\notin \overline{\Gamma}$, and $v\in C^2(\overline{B^+}\setminus\{0\})$ satisfy
\begin{equation}\label{lowerconicalequ}
    \begin{cases}
        \lambda(\av)\in \overline{\Gamma}~ \text{in}~B^+,\\
        \frac{\p v}{\p x_n}\leq c\cdot e^v~\text{on}~\p' B^+\setminus\{0\}.
    \end{cases}
\end{equation}  
Then $\liminf\limits_{x\to 0}$ $v(x)\in(-\infty,+\infty]$.
\end{proposition}

\begin{remark}\label{ms_starter-optimal-remark}
The assumption $\bl\notin\overline\Gamma$ in the above proposition is optimal in the sense that the conclusion fails for any cone $\Gamma$ with $\bl\in\overline{\Gamma}$ and any $c\in\bR$, as shown by Example \ref{reex1} at the end of this section.
\end{remark}
\begin{proposition}\label{notouching-thm-av-critical}
    For $n\geq 2$ and $c\in\bR$, let $(f,\Gamma)$ satisfy \eqref{eqn-230331-0110} and \eqref{eqn-240223-0305} with $\bl\notin\overline\Gamma$.
    Suppose that $u\in C^2(\overline{B^+}\setminus\{0\})$ and $v\in C^2(\overline{B^+})$ satisfy $\lambda(\av)(\overline{B^+})\subset\Gamma$,
\begin{equation}\label{notounchingproblem-av-critical}
    \begin{cases}
         f(\lambda(A[u])) \geq f(\lambda(\av))~~ \text{in}~{B^+},\\
        \frac{\p u}{\p x_n}\leq c \cdot e^u~~\text{on}~\p' B^+\setminus\{0\},~~
        \frac{\p v}{\p x_n}\geq c \cdot e^v~~  \text{on}~\p' B^+,
        \end{cases}
\end{equation}
and $u>v$ in $B^+$.
Then $\liminf\limits_{x\to 0} (u-v)(x)>0$.
\end{proposition}
Denote, as in \cite{CLL-1}, 
$v^{x,\lambda} := v^{\varphi_{x,\lambda}}$, where $\varphi_{x,\lambda} (y) := y + \lambda^2 (y-x)/|y-x|^2$ is the Kelvin transformation with respect to sphere $\p B_\lambda(x)$ in $\bR^n$.
\begin{proof}[Proof of Proposition \ref{ms_starter}]
    Let $u\coloneqq v^{0,1}$, then 
    \begin{equation}\label{240131-1631}
            \begin{cases}
        \lambda(A[u])\in \overline{\Gamma}~\text{in}~\hrn\setminus B^+,\\
        \frac{\p u}{\p x_n}\leq c\cdot e^u < (|c|+1)\cdot e^u~\text{on}~\p\hrn\setminus \p' B^+,
    \end{cases}
    \end{equation}
    and the desired conclusion is equivalent to $\liminf_{|x|\to 0}(u(x)+2\log|x|)>-\infty$.
    
    Since $\bl\notin\overline\Gamma$, we have $\mgn>1$. Let $1<\mu< \mgn$ and $\delta>0$  which will be sent to $0$. Consider 
    \begin{equation*}
        w_\delta(x)\coloneqq \frac{2}{\mu-1} \log \left( \epsi ( |x-\bm{e_n}|^{1-\mu}-\delta )\right),
    \end{equation*}
    where $\bm{e_n}=(0,\dots,0,1)$ and $\epsi>0$ is a sufficiently small constant to be determined later. 
    
    A computation as in \cite{CLL-1} gives
    \begin{equation*}
        \lambda(A[w_\delta])=C(r)(\mu,-1,\dots,-1),
    \end{equation*}
    where $r=|x-\bm{e_n}|$, $C(r)=2^{-1} e^{-2 w_\delta} w_{\delta}^{'}(w_\delta+2\log r)'>0$, and the primes denote derivatives on the variable $r$. It follows that 
    \begin{equation}\label{equ-w_delta}
    \lambda(A[{w_\delta}])\in \bR^n\setminus \overline\Gamma\quad \text{in}~ B_R(\bm{e_n})\setminus B_3(\bm{e_n}),
    \end{equation}
 where $R=\delta^{-{1}/{(\mu-1)}}$. 
    
    Next we select a small constant $\epsi>0$, which is independent of $\delta$, such that 
    \begin{equation}\label{epsigoal}
        \begin{cases}
            \frac{\p w_\delta}{\p x_n}\geq (|c|+1)\cdot e^{w_\delta}~ \text{on}~\p' B_R(\bm{e_n})\setminus \p' B_3(\bm{e_n}), \\
            w_\delta\leq u~\text{on}~\p'' B_3(\bm{e_n}).
        \end{cases}
    \end{equation}
    A direct computation gives
    \begin{equation*}
        \frac{\p w_\delta}{\p x_n}\cdot e^{-w_\delta}= 
        2 \epsi^{-\frac{2}{\mu-1}} (1-\delta r^{\mu-1})^{-\frac{\mu +1}{\mu -1}}
        ~\text{on}~\p' B_R(\bm{e_n})\setminus \p' B_3(\bm{e_n}),    \end{equation*}
    where $r=|x-\bm{e_n}|\in [3,R]$. Since the right hand side above is increasing in $r$, and $\delta$ is closed to $0$, we have
    \begin{equation*}
        \frac{\p w_\delta}{\p x_n}\cdot e^{-w_\delta}\geq 
        2 \epsi^{-\frac{2}{\mu-1}} (1- 3^{\mu-1}\delta )^{-\frac{\mu +1}{\mu -1}}     \geq \epsi^{-\frac{2}{\mu-1}}~\text{on}~\p' B_R(\bm{e_n})\setminus \p' B_3(\bm{e_n}).
    \end{equation*}
    On the other hand, $w_\delta \leq \frac{2}{\mu-1}\log \epsi$ on $\p''B_{3}(\bm{e_n})$. Combining the two inequalities, we can choose a desired $\epsi$ satisfying \eqref{epsigoal}.
    
    We next prove that 
    \begin{equation}\label{mp-argument}
    u\geq w_\delta\quad \text{in}~~\Omega_R\coloneqq B^+_R(\bm{e_n})\setminus \overline{B^+_3(\bm{e_n})}.
    \end{equation}

    Suppose the contrary, then, using \eqref{epsigoal} and $w_\delta(x)\to -\infty$ as $x\to \p B_R(\bm{e_n})$, we have   $a\coloneqq (w_\delta-u)(x_0)=\sup_{\Omega_R} (w_\delta-u)>0$ for some $x_0\in \Omega_R\cup \p'\Omega_R$. Let $\widetilde{w_\delta}\coloneqq w_\delta-a$, then
    $\widetilde{w_\delta}\leq u$ in $\overline{\Omega_R}$ and $\widetilde{w_\delta}=u$ at $x_0$.
    
    If 
    $x_0\in\Omega_R$, we have $\nabla^2 u(x_0)\geq \nabla^2 \widetilde{w_\delta}(x_0)$ and $\nabla u(x_0)=\nabla \widetilde{w_\delta}(x_0)$. By \eqref{eqn-230331-0110} and \eqref{240131-1631}, we have $\lambda(A[\widetilde{w_\delta}])(x_0)\in\overline\Gamma$. On the other hand, by \eqref{eqn-230331-0110} and \eqref{equ-w_delta}, we have 
$\lambda(A[\widetilde{w_\delta}])(x_0)\in \bR^n\setminus{\overline\Gamma}$. A contradiction. 

We are left to rule out the case $x_0\in\p'\Omega_R$. By the derivative tests, we have $\p_{x_n}u(x_0)\geq \p_{x_n}\widetilde{w_\delta}(x_0)$. By \eqref{epsigoal}, $\p_{x_n}\widetilde{w_\delta}(x_0)\geq (|c|+1)e^{w_\delta(x_0)}=(|c|+1)e^{u(x_0)+a}> (|c|+1)e^{u(x_0)}$. Hence, $\p_{x_n}u(x_0)>(|c|+1)e^{u(x_0)}$. A contradiction to \eqref{240131-1631}. Therefore, \eqref{mp-argument} is proved.

Sending $\delta$ to $0$ in \eqref{mp-argument}, we obtain
    \begin{equation*}
        u(x)\geq \frac{2}{\mu-1}\log \epsi -2\log|x-\bm{e_n}|,~x\in \hrn\setminus B_{3}^+(\bm{e_n}).
    \end{equation*}
    It follows that $\liminf\limits_{|x|\to\infty}(u(x)+2\log|x|)\geq \frac{2}{\mu-1}\log \epsi >-\infty$.  Proposition \ref{ms_starter} is proved.
\end{proof}

The proof of Proposition \ref{notouching-thm-av-critical} makes use of the following lemma.

\begin{lemma}\label{240204-1953}   
For $n\geq 2$ and $c\in\bR$, let $\Gamma$ satisfy \eqref{eqn-230331-0110} with $\bl\notin \overline{\Gamma}$,
$\delta>0$ be a constant,
and $v\in C^2(\overline{B^+}\setminus\{0\})$ satisfy \eqref{lowerconicalequ}. Then there exist a constant $\tau\geq 0$, depending only on $\widetilde{\delta}\coloneqq c\cdot e^{v(0)}+\delta$, and a sequence of numbers $r_j\to 0+$ such that $v(0)\coloneqq \liminf_{x\to 0} v(x)$ satisfies 
    \begin{equation}\label{240204-1955}
        v(0)
        \geq \inf\limits_{x\in \p'' B_{r_j}^+(0)} \left( v(x) - \widetilde{\delta} x_n -\frac{\tau}{2}|x|^2 \right),~\forall j.
    \end{equation}
    Moreover, when $c\leq 0$, inequality \eqref{240204-1955} holds for all $r>0$ sufficiently small, not only for a sequence $\{r_j\}$. 
\end{lemma}

\begin{proof}
     By Proposition \ref{ms_starter}, $v(0)\coloneqq \liminf_{x\to 0} v(x)>-\infty$. We may assume that $v(0)<\infty$ since otherwise the desired conclusion is obvious.
     
    Consider
      \begin{equation*}
         v_\alpha(x)\coloneqq \alpha \log |x|, ~v_0(x)\coloneqq \widetilde{\delta} x_n + \frac{\tau}{2}|x|^2, ~ \text{and}~  w_\alpha(x)\coloneqq v_\alpha(x)+v_0(x),
     \end{equation*}
     where $\tau=5\widetilde{\delta}^2$ and $\alpha>0$ is a constant which will be sent to $0$. 
     
     For a function $\varphi$, we denote $W[\varphi]\coloneqq e^{2\varphi}A[\varphi]=-\nabla^2\varphi+\nabla \varphi\otimes\nabla\varphi -\frac{1}{2}|\nabla \varphi|^2 I$. A direct computation yields
     \begin{equation*}
         W[w_\alpha]= W[v_0] + \left( W[v_\alpha] + E \right),
     \end{equation*}
     where 
     \begin{equation*}
         E=\nabla v_\alpha \otimes \nabla v_0 +\nabla v_0\otimes \nabla v_\alpha -(\nabla v_\alpha\cdot \nabla v_0) I= \alpha O(|x|^{-1}).
     \end{equation*}
     Here $|O(|x|^{-1})|\leq C|x|^{-1}$, where $C>0$ is independent of $\alpha$, but allowed to depend on $\widetilde{\delta}$. 
     
      By a calculation, $W[v_\alpha](x)=\alpha |x|^{-2}\{ (-1-\frac{\alpha}{2})I +(2+\alpha)\frac{x}{|x|}\otimes \frac{x}{|x|}\}$, and thus we have 
    \begin{equation*}
        \lambda(W[v_\alpha]+E)(x)= \alpha |x|^{-2}\left((1+\frac{\alpha}{2})\bl +O(|x|)\right).
    \end{equation*}
    Since $\bl\notin\overline{\Gamma}$ and $\alpha$ is closed to $0$, there exists $r_1>0$, depending only on $\widetilde{\delta}$, such that 
    \begin{equation}\label{240919-1528}
        \lambda(W[v_\alpha]+E)(x)\in \bR^n\setminus\overline{\Gamma},~0<|x|<r_1.         \end{equation}
    
    By a direct computation, 
     \begin{equation*}
         W[v_0](x)= -(\tau +\widetilde{\delta}^2/2) I +\widetilde{\delta}^2 \bm{e_n}\otimes \bm{e_n} +\tau \widetilde{\delta} O(|x|) +\tau^2 O(|x|^2),
     \end{equation*}
     where $\bm{e_n}=(0,\dots,0,1)$.
     By the above expression,
     there exists $r_2>0$, depending only on $\widetilde{\delta}$, such that
     \begin{equation}\label{240919-1529}
              W[v_0](x)\leq -\widetilde{\delta}^2 I\leq 0,~0<|x|<r_2.
    \end{equation}

     By \eqref{240919-1528}, \eqref{240919-1529}, and $\Gamma+\Gamma_n\subset \Gamma$, we have $\lambda(W[w_\alpha])(x)\in \bR^n\setminus\overline{\Gamma}$ for any $0<|x|<\widetilde{r}\coloneqq \min\{r_1,r_2\}$.
     Namely, $w_\alpha$ is a strict subsolution of the equation $\lambda(A[u])\in\p\Gamma$.

\medskip
     
     The remainder of the proof is divided into two cases.

\smallskip
     \textbf{Case 1:} $c\leq 0$. 
     
     Let $\{r_j\}$ be any sequence of numbers tending to $0$ with $0<r_j<\widetilde{r}$, $\forall j$, and 
    \begin{equation*}
         \ubar{v}_{\alpha}^{(j)}(x)\coloneqq w_\alpha (x) +\inf\limits_{x\in \p'' B_{r_j}^+(0)} \left(v(x)-\widetilde{\delta}x_n-\frac{\tau}{2}|x|^2\right).
    \end{equation*}
    Clearly, we have $\lambda(A[\ubar{v}_{\alpha}^{(j)}])\in \bR^n\setminus\overline{\Gamma}$ in $B_{\widetilde{r}}\setminus\{0\}$, as a strict subsolution,  and $\ubar{v}^{(j)}_{\alpha}\leq v$ on $\p'' B^+_{r_j}(0)$.
    
     We claim that for $j$ large,
     \begin{equation}\label{240206-1610}
        \ubar{v}_{\alpha}^{(j)}\leq v~\text{in}~\overline{B^+_{r_j}}\setminus\{0\}.
     \end{equation}
     
     Suppose the contrary that along a subsequence $j\to\infty$,
     \begin{equation*}
         \ubar{c}\coloneqq  -\inf_{x\in \overline{B^+_{r_j}}}(v-\ubar{v}_{\alpha}^{(j)})(x)>0.
     \end{equation*}
    Since $\ubar{v}_{\alpha}^{(j)}(x)\to -\infty$ as $x\to 0$ and $\ubar{v}^{(j)}_{\alpha}\leq v$ on $\p'' B^+_{r_j}(0)$, there exists some $\ubar{x}\in B_{r_j}^+\cup (\p' B_{r_j}^{+}\setminus\{0\})$ such that $\ubar{c}=-(v-\ubar{v}_{\alpha}^{(j)})(\ubar{x})$. Hence, $\ubar{v}\coloneqq \ubar{v}_{\alpha}^{(j)}-\ubar{c}$ touches $v$ from below at $\ubar{x}$, i.e. $\ubar{v}\leq v$ in $B^+_{r_j}$ and $\ubar{v}(\ubar{x})=v(\ubar{x})$.
    
    There are two possibilities.

    \noindent \textbf{Case 1.1:} $\ubar{x}\in B_{r_j}^+$. By the derivative tests, we have $\nabla \ubar{v}(\ubar{x})=\nabla v(\ubar{x})$ and $\nabla^2 \ubar{v}(\ubar{x})\leq \nabla^2 v(\ubar{x})$.
    Since $\lambda(A[v])\in\overline\Gamma$ in $B_{r_j}^+$ and $\Gamma+\Gamma_n\subset\Gamma$, we have $\lambda(A[\ubar{v}])(\ubar{x})\in\overline{\Gamma}$. On the other hand, by the cone property of $\Gamma$, $
    \lambda(A[\ubar{v}])=        e^{-2 \ubar{c}} \lambda(A[\ubar{v}_{\alpha}^{(j)}])    
    \in \bR^n\setminus\overline{\Gamma}$ in $B_{r_j}^+$, a contradiction.

    \noindent\textbf{Case 1.2:} $\ubar{x}\in\p' B_{r_j}^+\setminus\{0\}$. Using $\frac{\p v}{\p x_n}\leq c\cdot e^v$ on $\p'B_{r_j}^+\setminus\{0\}$, we have $\frac{\p \ubar{v}}{\p x_n}(\ubar{x})\leq 
    c\cdot e^{v(\ubar{x})}$. By $v(\ubar{x})\geq v(0)+o_j(1)$ and $c\leq 0$, we have 
    \begin{equation*}
        \frac{\p \ubar{v}}{\p x_n} (\ubar{x}) \leq c\cdot e^{v(0)}+o_j(1),
    \end{equation*}
    where $o_j(1)\to 0$ as $j\to\infty$.
    On the other hand, a direct computation gives 
    \begin{equation*}
        \frac{\p \ubar{v}}{\p x_n}(\ubar{x})= \widetilde{\delta}= \delta +c\cdot e^{v(0)}. 
    \end{equation*}
    We reach a contradiction for large $j$.

    We have proved the claim \eqref{240206-1610}.
    
    Sending $\alpha$ to $0$ in \eqref{240206-1610}, we obtain that
    \begin{equation*}
       \widetilde{\delta}x_n+ \frac{\tau}{2} |x|^2+\inf\limits_{x\in \p'' B_{r_j}^+(0)} \left(v(x)-\widetilde{\delta}x_n-\frac{\tau}{2}|x|^2\right)\leq v(x),~x\in \overline{B^+_{r_j}}\setminus\{0\}.
      \end{equation*}
     Then, sending $x$ to $0$ in the above, the desired conclusion of Lemma \ref{240204-1953} follows in this case.

\medskip

\textbf{Case 2:} $c> 0$.

      We select 
     $\{r_j\}$ to be a sequence of numbers tending to $0$ such that $v(0)= \lim_{j\to\infty}\inf_{\p'' B^+_{r_j}}v$ and $0<r_j<\widetilde{r}$, $\forall j$. Let $\ubar{v}_{\alpha}^{(j)}$ be defined the same as that in Case $1$. 
     As in Case 1, we have $\lambda(A[\ubar{v}_{\alpha}^{(j)}])\in \bR^n\setminus\overline{\Gamma}$ in $B_{\widetilde{r}}\setminus\{0\}$ and $\ubar{v}^{(j)}_{\alpha}\leq v$ on $\p'' B^+_{r_j}(0)$.

Next we prove that for large $j$, independent of $\alpha$, we have
     \begin{equation}\label{240204-1636}
             \frac{\p\ubar{v}_\alpha^{(j)}}{\p x_n}\geq c\cdot e^{\ubar{v}^{(j)}_{\alpha}}~\text{on}~\p' B_{r_j}^+\setminus\{0\}.
     \end{equation}
     
By a direct computation, for 
$x\in\p'B^+\setminus\{0\}$, we have
\begin{equation*}
     \frac{\p \ubar{v}_{\alpha}^{(j)}}{\p x_n}(x)= \widetilde{\delta}=\delta +c\cdot e^{v(0)},
\end{equation*}
and, using $c\geq 0$ and $|x|\leq 1$,
\begin{align*}
    c\cdot e^{\ubar{v}_{\alpha}^{(j)}(x)} &= c\cdot |x|^{\alpha} \exp \left( {\frac{\tau}{2}|x|^2 +\inf\limits_{x\in \p'' B_{r_j}^+(0)} ( v(x)-\widetilde{\delta}x_n-\frac{\tau}{2}|x|^2 )} \right)\\
    &\leq   
    c\cdot (1+o_j(1)) 
    \exp \left(\inf\limits_{x\in \p'' B_{r_j}^+(0)} v(x) \right)= c\cdot (1+o_j(1))e^{v(0)}.
\end{align*}
Combining the above, since $\delta>0$ is a fixed constant, estimate \eqref{240204-1636} is valid for large $j$. 

Now we prove that 
\eqref{240206-1610} holds in this case. 

Suppose the contrary, then, since $\ubar{v}_{\alpha}^{(j)}(x)\to -\infty$ as $x\to 0$ and $\ubar{v}^{(j)}_{\alpha}\leq v$ on $\p'' B^+_{r_j}(0)$, we have $a\coloneqq (\ubar{v}_{\alpha}^{(j)}-v)(\ubar{x}) =\sup_\Omega(\ubar{v}_{\alpha}^{(j)}-v)>0$ for some $\ubar{x}\in B_{r_j}^+\cup (\p' B_{r_j}^+\setminus\{0\})$. 
Let $\widetilde{\ubar{v}}_{\alpha}^{(j)}\coloneqq \ubar{v}_\alpha^{(j)}-a$, then $\widetilde{\ubar{v}}_{\alpha}^{(j)}\leq v$ in $\overline{B_{r_j}^+}$, $\widetilde{\ubar{v}}_{\alpha}^{(j)}(\ubar{x})= v(\ubar{x})$, and $\lambda(A[\widetilde{\ubar{v}}_{\alpha}^{(j)} ])\in\bR^n\setminus\overline{\Gamma}$ in $B_{r_j}^+$. 

If $\ubar{x}\in B_{r_j}^+$,  we have $\nabla^2 v(\ubar{x})\geq \nabla^2 \widetilde{\ubar{v}}_{\alpha}^{(j)}(\ubar{x})$ and $\nabla v(\ubar{x})=\nabla \widetilde{\ubar{v}}_{\alpha}^{(j)}(\ubar{x})$. By \eqref{eqn-230331-0110} and \eqref{lowerconicalequ}, we have $\lambda(A[\widetilde{\ubar{v}}_{\alpha}^{(j)}])(\ubar{x})\in\overline\Gamma$. A contradiction. 

We are left to rule out the case $\ubar{x}\in\p' B_{r_j}^+\setminus\{0\}$. By the derivative tests, we have $\p_{x_n}v(\ubar{x})\geq \p_{x_n}\widetilde{\ubar{v}}_{\alpha}^{(j)}(\ubar{x})$. By \eqref{240204-1636}, $\p_{x_n}v(\ubar{x})\geq c\cdot e^{\ubar{v}_{\alpha}^{(j)}(\ubar{x})}=c\cdot e^{\widetilde{\ubar{v}}_{\alpha}^{(j)}(\ubar{x})+a}> c\cdot e^{v(\ubar{x})}$. A contradiction with \eqref{lowerconicalequ}. Therefore, \eqref{240206-1610} is proved.

The desired conclusion of Lemma \ref{240204-1953} follows after sending $\alpha$ to $0$ and then $x$ to $0$ in \eqref{240206-1610}, as in Case 1.
\end{proof}

The following lemma will also be used in the proof of Proposition \ref{notouching-thm-av-critical}.
We denote the gradient operator as $\nabla =(\nabla_T, \p_{x_n})$, where $\nabla_T\coloneqq(\p_{x_1},\dots,\p_{x_{n-1}})$.

\begin{lemma}\label{vansihgrad}
    Let $n\geq 2$ and $v$ be a $C^1$ function defined in a neighborhood of $0\in \overline{\hrn}$ with $q \coloneqq \nabla_T v(0) \neq 0$.
Then there exists a M\"obius transformation $\psi$ such that 
\begin{equation*}
    \psi(0) = 0,~ \psi(\hrn)=\hrn,~ \text{and}~ \nabla_T v^\psi (0) = 0.
\end{equation*}
Moreover, all such $\psi$ can be written as
\begin{equation} \label{240201-1609}
    \psi(x) = O \left( \frac{\lambda^2 (x-\bar{x}) }{|x - \bar{x}|^2} + \frac{\lambda^2 \bar{x}}{|\bar{x}|^2} \right),
\end{equation}
where $\lambda \in \bR\setminus\{0\}$, $O=\operatorname{diag}\{O',1 \}$, $O'\in O(n-1)$, and $\bar{x} = \frac{\lambda^2}{2} O^T
     ( q^T, 0)^T$.
\end{lemma}

\begin{proof}
    First, from \cite[Theorem~3.5.1]{MR1393195}, all M\"obius tranforms mapping both zero to zero and $\hrn$ to $\hrn$ can be written as either \eqref{240201-1609} or $\psi(x) = \lambda^2 O x$ with $O=\operatorname{diag}\{O',1\}$, $O'\in O(n-1)$ and ${\bar x}_n=0$.  Clearly, the second possibility can be ruled out as $\nabla_T v^\psi(0) = \lambda^2 O^T q \neq 0$. Now, for $\psi$ in the form of \eqref{240201-1609}, a direct computation shows
\begin{equation*}
    \nabla v^\psi (0) 
    = 
    2 \frac{\bar{x}}{|\bar{x}|^2} 
    + 
    \frac{\lambda^2}{|\bar{x}|^2} \left( I_n - 2 \frac{\bar{x}}{|\bar{x}|} \otimes \frac{\bar{x}}{|\bar{x}|} \right) O^T \nabla v(0).
\end{equation*}
Matching $\nabla_T v^\psi (0) = 0$, we obtain the desired formula.  
\end{proof}

\begin{proof}[Proof of Proposition \ref{notouching-thm-av-critical}]
It is clear that $u(0)\coloneqq \liminf_{x\to 0}u(x)\geq v(0)$.
Suppose the contrary, then $u(0)=v(0)$. 

We only need to consider the case that $\nabla_T v(0)=0$, since otherwise, by Lemma \ref{vansihgrad}, there exists a M\"obius transformation $\psi$ in the form of \eqref{240201-1609} such that $\psi(0) = 0$, $\psi(\hrn)=\hrn$, and $\nabla_T v^\psi (0) = 0$. Then, by the M\"obius invariance of \eqref{notounchingproblem-av-critical}, we may work with $u^\psi$ and $v^\psi$ instead of $u$ and $v$.     
    
Let 
    \begin{equation*}
        v_\epsi (x)\coloneqq v(x) + \epsi(|x|+x_n),
    \end{equation*}
    where $\epsi>0$ is some sufficiently small constant. 

    Next we prove that \begin{equation}\label{critical--proof}
        \lambda(\epsi)\coloneqq -\inf\limits_{\overline{B^+}\setminus\{0\}} (u-v_\epsi) >0.
    \end{equation}

    Indeed, using the boundary conditions in \eqref{notounchingproblem-av-critical}, we have, for $x\in \overline{B^+}\setminus\{0\}$,
    \begin{align}\label{240913-2145}
        u(x)-v_\epsi(x)=u(x)-\left(v(0)+(\frac{\p v}{\p x_n}(0)+\epsi)x_n+\epsi|x|+o(|x|)\right) \notag \\
        \leq u(x) -u(0) -(c\cdot e^{u(0)}+\epsi)x_n -\epsi |x| +o(|x|).
    \end{align}
    By Lemma \ref{240204-1953} (with $v=u$ and $\delta=\epsi$ there), there exists a sequence of points $\{x^j\}\subset\overline{B^+}\setminus\{0\}$ with $x^j\to 0$ as $j\to \infty$ such that
    \begin{equation*}
        u(x^j)-u(0)-(c\cdot e^{u(0)}+\epsi)x^j_n +o(|x^j|)\leq 0,\quad\forall j.
    \end{equation*}
    Combining this with \eqref{240913-2145}, we have
    \begin{equation*}
        u(x^j)-v_\epsi(x^j)\leq -\frac{\epsi}{2} |x^j| <0,\quad \text{for large}~j.
    \end{equation*}
    Hence, \eqref{critical--proof} is proved.

    Since $\liminf_{x\to 0}(u-v_\epsi)(x)=(u-v)(0)=0$, there exists $x_\epsi\in \overline{B^+}\setminus\{0\}$ such that
    \begin{equation}\label{perturbtouching-av}
        u-\widetilde{v_\epsi}\geq 0~\text{in}~\overline{B^+}\setminus\{0\},~\text{and}~(u-\widetilde{v_\epsi})(x_\epsi)=0,
    \end{equation}
    where 
    \begin{equation*}
        \widetilde{v_\epsi}(x)\coloneqq v_\epsi(x)-\lambda(\epsi)=v(x)+\epsi(|x|+x_n)-\lambda(\epsi).
    \end{equation*}
Using the positivity of $u-v$ in $\overline{B^+}\setminus\{0\}$, we deduce from the above that
\begin{equation}\label{240130-1357-av}
    \lambda(\epsi)= -(u-v_\epsi)(x_\epsi)= v(x_\epsi)-u(x_\epsi)+\epsi (|x_\epsi|+(x_\epsi)_n)\leq 2\epsi |x_\epsi|,
\end{equation}
and 
\begin{equation}\label{240130-1348-av}
    \lim\limits_{\epsi\to 0}|x_\epsi|=0.
\end{equation}

We will show, for small $\epsi>0$, that the following holds:
\begin{equation}\label{240125-2006-av-1}
\text{either}~f(\lambda(\av))(x_\epsi)>f(\lambda(A[\widetilde{v_\epsi}]))(x_\epsi)~
    \text{or}~\lambda(A[\widetilde{v_\epsi}])(x_\epsi)\notin\Gamma,     
\quad \text{if}~x_\epsi\in B^+,
\end{equation}
and
\begin{equation}\label{240125-2006-av-2}
        \frac{\p v}{\p x_n}(x_\epsi)+\epsi =\frac{\p\widetilde{v_\epsi}}{\p x_n}(x_\epsi),\quad \text{if}~x_\epsi\in \p' B^+\setminus \{0\}.
\end{equation}
These would lead to a contradiction. Indeed, if $x_\epsi\notin \p' B^+$, equation \eqref{notounchingproblem-av-critical} and \eqref{perturbtouching-av} imply that 
\begin{equation*}
    f(\lambda(A[\widetilde{v_\epsi}]))(x_\epsi)   \geq
    f(\lambda(\av))(x_\epsi).
\end{equation*}
A contradiction to \eqref{240125-2006-av-1}.
On the other hand, if $x_\epsi\in \p' B^+\setminus\{0\}$, by \eqref{notounchingproblem-av-critical} and \eqref{perturbtouching-av}--\eqref{240130-1348-av},
\begin{equation*}
\frac{\p \widetilde{v_\epsi}}{\p x_n}(x_\epsi) \leq c \cdot e^{\widetilde{v}_\epsi(x_\epsi)}= c\cdot e^{v(x_\epsi)}+o(\epsi)
    \leq \frac{\p v}{\p x_n}(x_\epsi)+o(\epsi).
\end{equation*}
The above contradicts to \eqref{240125-2006-av-2} when $\epsi$ is small.

To complete the proof, we show the validity of \eqref{240125-2006-av-1} and \eqref{240125-2006-av-2} for small $\epsi$. Assertion \eqref{240125-2006-av-2} follows from a direct computation. Assertion \eqref{240125-2006-av-1} can be proved similarly to \cite[Proof of Theorem 1.1]{CLN1}. We may assume that $\lambda(A[\widetilde{v_\epsi}])(x_\epsi)\in\Gamma$ for any $\epsi$ small. By a computation and \eqref{240130-1357-av}, we have 
\begin{equation}\label{240212-1253-av}
    |\widetilde{v_\epsi}- v|(x_\epsi)+|\nabla (\widetilde{v_\epsi}-v)|(x_\epsi)=O(\epsi),
\end{equation}
and, for some $O(x)\in O(n)$, 
\begin{equation}\label{240212-1527}
    \nabla^2(\widetilde{v_\epsi}- v)(x)={\epsi}{|x|^{-1}} O(x)^t \cdot \operatorname{diag}\{1,\dots,1,0\}\cdot  O(x).
\end{equation}

Denote 
\begin{equation*}
F(s,p,M)\coloneqq f(\lambda(e^{-2s}(-M+p\otimes p-2^{-1}|p|^2 I))).
\end{equation*}
By \eqref{eqn-240223-0305} and $\lambda(\av)(\overline{B^+})\subset\Gamma$, there exists $b,\delta>0$ such that the set 
\begin{equation*}
    S\coloneqq \{(s,p,M)\mid x\in B^+_{1/2},~\|(s,p,M)-(v,\nabla v,\nabla^2 v)(x)\|\leq \delta \}
\end{equation*}
satisfies 
\begin{equation*}
   -F_{M_{ij}}(s,p,M)\geq b\delta_{ij}\quad \text{a.e.}~(s,p,M)\in S. 
\end{equation*}

\noindent Let
\begin{equation*}
    \overline{t_\epsi}\coloneqq
    \begin{cases}
        0\quad \text{if}~ \|\nabla^2 (\widetilde{v_\epsi}-v)(x_\epsi)\|\leq \delta/5, \\
        1- \delta/(5\|\nabla^2 (\widetilde{v_\epsi}- v)(x_\epsi)\|)\quad \text{if}~ \|\nabla^2 ( \widetilde{v_\epsi}- v)(x_\epsi)\|> \delta/5.
    \end{cases}
\end{equation*}
Then by \eqref{240212-1253-av} and the definition of $\overline{t_\epsi}$,
\begin{equation*}
    ( \widetilde{v_\epsi},\nabla \widetilde{v_\epsi}, \nabla^2 v+(1-t)(\nabla^2\widetilde{v_\epsi}-v))(x_\epsi)\in S,~\forall \overline{t_\epsi}\leq t\leq 1.
\end{equation*}

Now, by \eqref{240212-1253-av}, there exists some $C>0$ independent of $\epsi$,
\begin{align}
    F( v, \nabla v, \nabla^2 v)(x_\epsi)& \geq F(\widetilde{v_\epsi}, \nabla \widetilde{v_\epsi}, \nabla^2 v)(x_\epsi) -C\epsi \notag \\
    &= F( \widetilde{v_\epsi}, \nabla \widetilde{v_\epsi}, \nabla^2 v)(x_\epsi)    - F( \widetilde{v_\epsi}, \nabla \widetilde{v_\epsi}, \nabla^2 \widetilde{v_\epsi})(x_\epsi)  \notag \\ 
    & \quad  + F( \widetilde{v_\epsi}, \nabla \widetilde{v_\epsi}, \nabla^2 \widetilde{v_\epsi}  )(x_\epsi) -C\epsi. \label{240212-1630}  
\end{align}
By the mean value theorem, 
\begin{align*}
 F( \widetilde{v_\epsi}, & \nabla \widetilde{v_\epsi}, \nabla^2 v)(x_\epsi)    - F( \widetilde{v_\epsi}, \nabla \widetilde{v_\epsi}, \nabla^2 \widetilde{v_\epsi})(x_\epsi) \\
 & = \int_0^1 (-F_{M_{ij}})( \widetilde{v_\epsi}, \nabla \widetilde{v_\epsi}, \nabla^2 {v}   + (1-t)\nabla^2( \widetilde{v_\epsi}-v))(x_\epsi) dt \cdot  (\widetilde{v_\epsi}-v)_{ij}(x_\epsi).
\end{align*}
Note that $\nabla^2 {v}   + (1-t)\nabla^2( \widetilde{v_\epsi}-v)=\nabla^2 \widetilde{v_\epsi}+t(\nabla^2 v-\nabla^2 \widetilde{v_\epsi})$, the above is justified by \eqref{240212-1527}, $\lambda(A[\widetilde{v_\epsi}])(x_\epsi)\in \Gamma$, and $\Gamma+\Gamma_n\subset\Gamma$.

By the definition of $\overline{t_\epsi}$ and \eqref{240212-1527}, we have 
 \begin{align*}
 F( \widetilde{v_\epsi}, & \nabla \widetilde{v_\epsi}, \nabla^2 v)(x_\epsi)    - F( \widetilde{v_\epsi}, \nabla \widetilde{v_\epsi}, \nabla^2 \widetilde{v_\epsi})(x_\epsi) \\
 & \geq  \int_{\overline{t_\epsi}}^1 (-F_{M_{ij}})( \widetilde{v_\epsi}, \nabla \widetilde{v_\epsi}, \nabla^2 {v}   + (1-t)\nabla^2( \widetilde{v_\epsi}-v))(x_\epsi) dt \cdot  (\widetilde{v_\epsi}-v)_{ij}(x_\epsi) \\ 
 &\geq b (1-\overline{t_\epsi}) \cdot  \Delta (\widetilde{v_\epsi}-v) (x_\epsi) \geq b(1-\overline{t_\epsi}) \|\nabla^2   (\widetilde{v_\epsi}-v) (x_\epsi) \| \\
 &\geq b \min\{\|\nabla^2   (\widetilde{v_\epsi}-v) (x_\epsi) \|, \delta/5\}\geq b \min\{\epsi/|x_\epsi|,\delta/5\}.
 \end{align*}
Therefore, the assertion \eqref{240125-2006-av-1} follows from \eqref{240130-1348-av}, \eqref{240212-1630} and the above, when $\epsi$ is sufficiently small. Proposition \ref{notouching-thm-av-critical} is proved.
\end{proof} 
The proof of Theorem \ref{nondegenerateliouville-critical} also makes use of the following lemma.
\begin{lemma}\label{diskradial}
   For $n\geq 2$, let $(f,\Gamma)$ satisfy \eqref{eqn-230331-0110} and \eqref{eqn-240223-0305}. Assume that $v\in C^2(B)$ is radially symmetric and satisfies $f(\lambda(\av))=1$ in $B$. Then
   \begin{equation*}
       v(x)\equiv \log \left( a/{(1+b|x|^2)} \right) \quad\text{in}~B,
   \end{equation*}
   where $a,b>0$ satisfy $f(2b a^{-2}\bm{e})=1$ with $\bm{e}=(1,\dots,1)$. 
\end{lemma}

 In dimensions $n\geq 3$, when $f$ is $C^1$, Lemma \ref{diskradial} was previously known. See \cite[Theorem 3]{Li-Li_JEMS} (with $v={2}{(n-2)^{-1}}\log u$ there).
 \begin{remark}
      When $\Gamma\subsetneqq\bR^n$ is a non-empty open symmetric set satisfying $\Gamma+\Gamma_n\subset\Gamma$, not necessarily being a cone, 
     the above Lemma still holds with $b>0$ being replaced by $b\geq-1$. This can be seen easily from the proof.
 \end{remark}

\begin{proof}
    The proof is somewhat simpler than that of 
    \cite[Theorem 3]{Li-Li_JEMS}. 
    In the following, we write $r=|x|$ and do not distinguish between $v(r)$ and $v(x)$.  
    Denote the eigenvalues of $\av$ by $\lambda^v=(\lambda^v_1,\dots,\lambda^v_n)$.
    As in the proof of \cite[Lemma 3.1]{CLL-1}, \begin{equation}\label{radialeigenvalue}
        \lambda_1^v= e^{-2 v}(-v''+2^{-1}(v')^2),~\lambda^v_i=e^{-2v}(-r^{-1}v'- 2^{-1}(v')^2),~i=2,\dots,n.
    \end{equation}
Since $v\in C^2(B)$, we must have $v'(0)=0$. Hence, $\lambda_i^v(0)\leq 0$ for $i=2,\dots,n$. By \eqref{eqn-230331-0110} and $\lambda(A[v])(0) \in \Gamma$, we have $\lambda_1^v(0)> 0$, and thus $v''(0) < 0$.

Let 
    \begin{equation*}
        w(x)=\log \left( a/{(1+b|x|^2)} \right),~a=e^{v(0)},~\text{and}~b=-2^{-1}v''(0).  \end{equation*}
    With the above choices of $a$ and $b$, we have
$w(0)=v(0)$, $w'(0)=v'(0)=0$, $w''(0)=v''(0)$, and $\lambda^w(x)\equiv 2b a^{-2} \bm{e}=\lambda^v(0)$, $x\in \bR^n$,
where $\lambda^w=(\lambda^w_1,\dots,\lambda^w_n)$ denotes eigenvalues of $A[w]$. In particular, $f(\lambda(A[w]))=1$ in $\bR^n$. 

Let
\begin{equation*}
    I := \{t\in (0,1]\mid v = w \,\,\text{in}\,\, [0, t] \}.
\end{equation*}
In the following, we prove $I = (0,1]$ by the continuity method.

Firstly, we prove $I \neq \emptyset$. 

Since $\lambda^v(0)=2ba^{-2}\bm{e}$ and $v\in C^2$, there exists some small $\widetilde{r}>0$, such that $\lambda^v(r) \in \overline{B_{ba^{-2}} (2ba^{-2}\bm{e})} \subset \Gamma_n \subset \Gamma$ for any $r < \widetilde{r}$. 
Condition \eqref{eqn-240223-0305} implies that there exist some $\gamma,\nu  \in \Gamma_n$ such that $\gamma \cdot (\lambda-\mu)\leq f(\lambda)-f(\mu)\leq \nu \cdot (\lambda-\mu)$ for all $\lambda,\mu\in \overline{B_{ba^{-2}} (2ba^{-2}\bm{e})} $.
Plugging $\lambda=\lambda^v(r)$ and $\mu=\lambda^w$ in this inequality and using the equations of $v$ and $w$, we obtain that for $r\in(0,\widetilde{r})$,
\begin{equation*}
    \gamma_1(\lambda^v_1(r)-\lambda^w_1)+(\textstyle\sum_{i\geq 2}\gamma_i ) (\lambda_2^v(r)-\lambda_2^w)\leq 0\leq \nu_1 (\lambda^v_1(r)-\lambda^w_1)+(\sum_{i\geq 2}\nu_i ) ( \lambda_2^v(r)-\lambda_2^w).
\end{equation*}
Hence, there exist some constant $c>0$ and some function $b$ on $r$ such that $c \leq b(r) \leq c^{-1}$ for $r\in (0,\widetilde{r})$ and 
\begin{equation} \label{eqn-240223-0358}
    \lambda_1^v(r) - \lambda_1^w(r) = -b(r) (\lambda_2^v (r) - \lambda_2^w (r) ) , \quad \forall r\in (0,\widetilde{r}).
\end{equation}

Using formula \eqref{radialeigenvalue} and $v'(0)=w'(0)$, we have 
\begin{equation*}
    \lambda^v_1(r)-\lambda^w_1(r)= -e^{-2v(r)}(v''(r)-w''(r)) +O(1)(|v(r)-w(r)|+|v'(r)-w'(r)|),
\end{equation*}
and
\begin{equation*}
    \lambda^v_2(r)-\lambda^w_2(r)= -e^{-2v(r)}{(v'(r)-w'(r))}/{r} +O(1)(|v(r)-w(r)|+|v'(r)-w'(r)|).
\end{equation*}
Substituting back to \eqref{eqn-240223-0358}, we have
\begin{equation} \label{eqn-240223-0403}
    v''(r)-w''(r)= - b(r) {(v'(r) - w'(r))}/{r} + O(1)|v(r)-w(r)|,~\forall 0<r<\widetilde{r},
\end{equation}
where $|O(1)|\leq C$. Now we show $v=w$ near $0$. From \eqref{eqn-240223-0403} and $v'(0) = w'(0)$, we have
\begin{equation*}
    v'(r) - w'(r) = \int_0^r e^{-\int_s^r b(t)/t\,dt} O(1) |v(s) - w(s)|\,ds.
\end{equation*}
Recall that $v(0) = w(0)$. Integrating again and using $b(t) > 0$, we have
\begin{equation*}
    |v(r) - w(r)| \leq C r^2 \sup_{s\in [0,r]} |v(s) - w(s)|.
\end{equation*}
This implies $v = w$ in a small neighborhood of $r=0$. Hence, we have proved $I \neq \emptyset$.

Next we prove $I$ is a relatively open subset of $(0,1]$. More precisely, we prove that for any $R \in I$ with $R<1$, there exists some $\epsi>0$ such that $v=w$ on $(0, R+\epsi)$.

By the definition of $I$ and $v, w\in C^2$, we have $v = w$, $v' = w'$ and $v'' = w''$ in $(0,R]$. Hence, again due to $v\in C^2$, there exists some small $\epsi_1>0$ such that $\lambda^v(r) \in B_{ba^{-2}} (2ba^{-2}\bm{e}) \subset \Gamma_n \subset \Gamma$ for any $r \in (R, R+\epsi_1)$. A computation as before yields, instead of \eqref{eqn-240223-0403},
\begin{equation*}
    v''(r) - w''(r) = O(1) (|v'(r) - w'(r)| + |v(r) - w(r)|), \quad \forall r\in (R,R+\epsi_1).
\end{equation*}
Integrating twice and arguing as before, we have $v=w$ in a small neighborhood of $r= R$.

Finally, by the definition of $I$ and $v, w\in C^0$, clearly $I$ is relatively closed. Hence, $I = (0,1]$, i.e. $v=w$ in $B$. The lemma is proved.
\end{proof}

\begin{proof}[Proof of Theorem \ref{nondegenerateliouville-critical}]
\textbf{Step 1:} Prove that for any $x\in \p \hrn$, there exists $\lambda_0(x)>0$ such that 
\begin{equation}\label{step1}
    v^{x,\lambda}\leq v ~\text{on}~\hrn\setminus B_\lambda (x),~\forall 0<\lambda<\lambda_0(x).
\end{equation}

The crucial part of this step is to show that
\begin{equation}\label{starter}
    \liminf_{|x|\to \infty} (v(x)+2\log |x|)>-\infty.
\end{equation}
Since $\lambda(\av)\in \Gamma$ in $\hrn$ and $\p_{x_n}v=c\cdot e^v$ on $\p\hrn$, we know that, by the M\"obius invariance, $\lambda(A[v^{0,1}])\in \Gamma$ in $B^+$ and ${\p_{x_n} v_{0,1}}=c\cdot e^{v_{0,1}}$ on $\p'B^+\setminus\{0\}$. Applying Proposition \ref{ms_starter} to $v^{0,1}$ yields $\liminf_{x\to 0}v^{0,1}(x)>-\infty$, i.e. \eqref{starter} holds.

Assertion \eqref{step1} can be proved by a similar argument to \cite[Proof of Lemma 2.1]{MR2001065}. For reader's convenience, we include a proof here.
Without loss of generality, take $x=0$, and denote $v^{0,\lambda}$ by $v^\lambda$. By the regularity of $v$, there exists $r_0>0$ such that 
\begin{equation*}
    \frac{d}{dr}(v(r,\theta)+\log r)>0,~\forall 0<r<r_0,~\theta\in S^{n-1}\cap \hrn.
\end{equation*}
It follows from the above that 
\begin{equation}\label{240212-1738}
    v^{\lambda}(y)< v(y),~\forall 0<\lambda<|y|<r_0,~y\in \hrn.
\end{equation}
Because of \eqref{starter}, there exists some constant $\alpha\in \bR$ such that 
\begin{equation*}
    v(y)\geq \alpha-2\log|y|,~\forall |y|\geq r_0,~y\in\hrn.
\end{equation*}
Let $\lambda_0\coloneqq \min\{\exp(\frac{1}{2}(\alpha-\max\limits_{\overline{B^+_{r_0}}}v)),r_0\}$. Then 
\begin{equation*}
    v^\lambda(y)\leq 2\log(\frac{\lambda_0}{|y|})+\max_{\overline{B^{+}_{r_0}}} v \leq \alpha -2\log|y|\leq v(y),~\forall 0<\lambda<\lambda_0,~y\in\hrn\setminus B_{r_0}.
\end{equation*}
Combining the above and \eqref{240212-1738}, Step $1$ is completed.

\medskip

By Step 1,
\begin{equation*}
    \bar\lambda(x)\coloneqq \sup\{\mu>0\mid v^{x,\lambda}\leq v ~\text{on}~\overline{\hrn}\setminus B_\lambda(x),~\forall 0<\lambda<\mu\}\in (0,\infty],~\forall x\in\p\hrn.
\end{equation*}

\textbf{Step 2:} Prove that if $\bar\lambda(x)<\infty$ for some $x\in\p\hrn$, then $v^{x,\bar\lambda(x)}\equiv v$ in $\overline \hrn\setminus\{x\}$.

    Without loss of generality, we take $x=0$, and denote $\bar{\lambda}=\bar{\lambda}(0)$ and $v^{\bar{\lambda}}=v^{0,\bar{\lambda}}$. By the definition of $\bar{\lambda}$ and the continuity of $v$,
    \begin{equation}\label{step2}
        v^{\bar{\lambda}}\leq v ~~\text{in}~\overline{\hrn}\setminus B_{\bar{\lambda}}.
    \end{equation}  
Note that from \eqref{halfspace-equ-v-critical},
   \begin{equation*}
f(\lambda(A[v]))= 1
        ~\text{in}\,\, \overline\hrn \setminus {B_{\bar{\lambda}}},
        ~\text{and}~
        \frac{\p v}{\p x_n}=c\cdot e^{v}~\text{on}~\p\hrn\setminus B_{\bar{\lambda}}.
    \end{equation*}    
A calculation gives, using \eqref{halfspace-equ-v-critical} and its M\"obius invariance,
\begin{equation}\label{eqn-230528-1202}
f(\lambda(A[v^{\bar{\lambda}}]))= 1
        ~\text{in}\,\, \overline\hrn \setminus {B_{\bar{\lambda}}},
        ~\text{and}~
        \frac{\p v^{\bar{\lambda}}}{\p x_n}=c\cdot e^{v^{\bar{\lambda}}}~\text{on}~\p\hrn\setminus B_{\bar{\lambda}}.
    \end{equation}

Hence, by the strong maximum principle and the Hopf Lemma (see the proof of \cite[Lemma 3]{Li-Li_JEMS} for details), we may assume 
\begin{equation} \label{eqn-230528-1207}
    v > v^{\bar{\lambda}} \,\,\text{in}\,\,\overline\hrn \setminus \overline{B_{\bar{\lambda}}}
\end{equation}
since otherwise $v^{\bar{\lambda}} \equiv v$ on $\overline\hrn \setminus {B_{\bar{\lambda}}}$, which gives the desired conclusion of Step 2.
Moreover, an application of the Hopf Lemma and \cite[Lemma 10.1] {MR2001065} as that in the proof of \cite[Lemma 3]{Li-Li_JEMS} gives 
\begin{equation} \label{eqn-230528-1212}
    \frac{\p}{\p \nu} (v - v^{\bar{\lambda}}) > 0
    \quad
    \text{on}\,\,
    \p B_{\bar{\lambda}}\cap \overline\hrn,
\end{equation}
where $\nu$ is the unit outer normal to $\p B_{\bar\lambda}$.

Next we prove
\begin{equation} \label{eqn-230528-1209}         \liminf_{|y|\rightarrow \infty} (v-v^{\bar{\lambda}})(y) > 0.
\end{equation}

Similar to \eqref{eqn-230528-1202}, we have
     \begin{equation*}
f(\lambda(A[v^{\bar{\lambda}}]))\geq 1
        ~\text{in}\,\, \overline{B^+_{\bar{\lambda}}},
        ~\text{and}~
        \frac{\p v^{\bar{\lambda}}}{\p x_n}=c\cdot e^{v^{\bar{\lambda}}}~\text{on}~\p' B^{+}_{\bar{\lambda}}\setminus \{0\}.
    \end{equation*}
    Applying Proposition \ref{notouching-thm-av-critical}, we have 
    \begin{equation}\label{38--reformulate}
        \liminf_{x\to 0}(v^{\bar{\lambda}}-v)(x)>0,
    \end{equation}
i.e. \eqref{eqn-230528-1209} holds.
    It follows from \eqref{eqn-230528-1207}--\eqref{eqn-230528-1209} (see the proof of \cite[Lemma 2.2]{MR2001065} and 
    \cite[Lemma 3.2]{Li2021ALT}) that for some $\epsi>0$, $v^{0,\lambda}(y)\leq v(y)$, for every $y\in\overline\hrn\setminus B_{\lambda}^+$ and $\bar{\lambda} \leq \lambda < \bar{\lambda} + \epsi$, violating the definition of $\bar{\lambda}$.

\medskip

\textbf{Step 3:} We prove that either $\bar{\lambda}(x)<\infty$ for all $x\in \p\hrn$ or $\bar{\lambda}(x)=\infty$ for all $x\in\p\hrn$.

Suppose for some $x\in \p\hrn$, $\bar{\lambda}(x)=\infty$. By definition, $v^{x,\lambda}(y) \leq v(y)$ on $\overline\hrn \setminus B_{\lambda}(x)$ for all $\lambda>0$. Adding both sides by $2\log |y|$ and sending $|y|$ to $\infty$, we have $v(x)+2\log \lambda \leq \liminf_{|y|\rightarrow \infty} (v(y)+2 \log|y|) $ for all $\lambda>0$. Hence, $\lim_{|y|\rightarrow \infty} (v(y)+2\log|y|) = \infty$.

On the other hand, if $\bar{\lambda}(x) < \infty$ for some $x\in\p\hrn$, we have, by Step 2, $v^{x,\bar{\lambda}(x)}\equiv v$ in $\hrn\setminus\{x\}$. Therefore, $\lim_{|y|\rightarrow \infty} (v(y)+2\log|y|) = v(x)+2\log \bar{\lambda}(x) < \infty$. Step 3 is proved.

\medskip

\textbf{Step 4:} We prove $\bar{\lambda}(x)<\infty$ for all $x\in\p\hrn$.

If not, by Step 3, $\bar{\lambda}(x)=\infty$ for all $x\in\p\hrn$.
Consequently, $v^{x,\lambda} (y) \leq v(y)$ for all $x\in\p\hrn$, $y\in\hrn\setminus B_{\lambda}(x)$, and $\lambda>0$. 
 This implies, by \cite[lemma 11.3]{MR2001065}, that 
$v(x',x_n)\equiv v(0',x_n)$, $\forall x'\in\bR^{n-1}$, $x_n\geq 0$.

In the following, we do not distinguish  $v(x)$ with $v(x_n)$, and write $v'$ for $dv/dx_n$. 

A direct computation gives that 
\begin{equation}\label{h-av-eigenvalue-critical}
    \lambda(\av)= (\lambda_1,\lambda_2,\dots,\lambda_2),
\end{equation}
where $\lambda_1\coloneqq -v''e^{-2v}+2^{-1}(v')^2 e^{-2v}$ and $\lambda_2=-2^{-1}(v')^2 e^{-2v}$. 

Since $\bl\notin\overline\Gamma$, i.e. $\mgn>1$, and $\lambda(\av)\in\Gamma$ in $\overline\hrn$, we have $\lambda_1+\mgn\lambda_2>0$ in $[0,\infty)$. Hence, denote $\mgn\coloneqq \delta+1$ for some $\delta>0$,
\begin{equation}\label{h2-critical}
   v''<-2^{-1}\delta (v')^2~~\text{in}~[0,\infty).
\end{equation} 
In particular, $ v''<0$ in $[0,\infty)$. 
The above also implies that
\begin{equation}\label{h1-critical}
    v'> 0~~\text{in}~[0,\infty).
\end{equation}
Indeed, if $v'\leq 0$ for some $\bar{x}_n\geq 0$, then, by \eqref{h2-critical}, we must have $v'<0$ in $(\bar{x}_n,\infty)$.
Integrating \eqref{h2-critical} twice, we obtain that 
$v(x_n)\leq b+ 2\delta^{-1}\log|2^{-1}\delta (x_n-\bar{x}_n-1)+(v')^{-1}(\bar{x}_n+1)|$ for some $b\in\bR$. This is a contradiction since the right hand side tends to $-\infty$ as $x_n\to -2 (\delta v'(\bar{x}_n+1))^{-1}+\bar{x}_n+1$. Hence, \eqref{h1-critical} is proved.
 
By $v'(0)=c\cdot e^{v(0)}$ and \eqref{h1-critical},
we are left to rule out the possibility of $c>0$. Integrating \eqref{h2-critical}, we obtain $v'(x_n)\leq 1/(2^{-1}\delta x_n +(v')^{-1}(0))$ for $x_n\geq 0$. In particular, we have
\begin{equation}\label{hclaim-critical}
   v'(x_n)\to 0\quad \text{as}~~x_n\to\infty.
\end{equation}

By \eqref{h-av-eigenvalue-critical}, \eqref{h1-critical}, and \eqref{hclaim-critical}, 
\begin{equation*}
    \lambda(\av)(x_n)=(-v''(x_n) e^{-2 v(x_n)},0,\dots,0)+o(1)\quad \text{as}~~ x_n\to\infty.
\end{equation*}
By equation \eqref{halfspace-equ-v-critical} and condition \eqref{natural-assumption}, there exists some constant $\delta_0>0$ such that 
\begin{equation*}
    -v''(x_n)e^{-2 v(x_n)}\geq \delta_0\quad \text{for}~ x_n~\text{large}. 
\end{equation*}
 By \eqref{h1-critical}, we have $-v''(x_n)\geq \delta_0 e^{2 v(0)}$. Integrating this gives $v'(x_n)\leq a-\delta_0 e^{2 v(0)} x_n$ for some $a\geq 0$. A contradiction to \eqref{h1-critical} for $x_n$ large. Step 4 is proved.

\medskip

Now, by Step 2 and 4, $\bar{\lambda}(x)<\infty$ for all $x\in\p\hrn$, and thus 
\begin{equation}\label{240216-1706}
    v^{x,\bar{\lambda}(x)} \equiv v~\text{in}~\overline\hrn\setminus\{x\},~\forall x\in\p\hrn.
\end{equation}
It follows, by \cite[Lemma 2.5]{Li-Zhu} (see also \cite[Lemma A.2]{Li2021ALT}), that
\begin{equation}\label{240216-1707}
    v(x',0)=\log \left(\frac{\hat{a}}{|x'-\bar{x}'|^2+d^2} \right),~x\in\p\hrn, 
\end{equation}
for some $\bar{x}\in\p\hrn$ and $\hat{a},d>0$.

Consider spheres $\p B_{\bar{\lambda}(x)}(x)$ for $x\in\p\hrn$. Using \eqref{240216-1706} at $\infty$ and the explicit formula \eqref{240216-1707}, we obtain that $P,Q\in \p B_{\bar{\lambda}(x)}(x)$, $x\in\p\hrn$, where $P=(\bar{x}',-d)$ and $Q=(\bar{x}',d)$. Consider the M\"obius transform 
\begin{equation*}
\psi(z)\coloneqq P+\frac{4 d^2 (z-P)}{|z-P|^2}~ \text{and} 
~w(z)\coloneqq v^{P,2d}(z)=2\log\left( \frac{2d}{|z-P|} \right) + v\circ\psi (z).
\end{equation*}
By the properties of M\"obius transforms, $\psi(\hrn)=B_{2d}(Q)$ and $\psi$ maps every sphere $\p B_{\bar{\lambda}(x)}(x)$, $x\in\p\hrn$, to every hyperplane through $Q$. By \eqref{240216-1706}, $w$ is symmetric with respect to all hyperplanes through $Q$, and thus $w$ is radially symmetric about $Q$ in $B_{2d}(Q)$.

By equation \eqref{halfspace-equ-v-critical} and its M\"obius invariance, 
\begin{equation*}
    f(\lambda(A[w]))=1\quad \text{in}~B_{2d}(Q).
\end{equation*}
By Lemma \ref{diskradial},
\begin{equation*}
    w(z)\equiv \log \big( \frac{\bar{a}}{1+\bar{b}|z-Q|^2} \big)\quad \text{in}~B_{2d}(Q).
\end{equation*}
Since $v(x)=w^{P,2d}(x)$, together with \eqref{240216-1707}, we have
\begin{equation*}
    v(x',x_n)\equiv \log \big(\frac{a}{1+b|(x',x_n)-(\bar x',\bar x_n )|^2}\big),
\end{equation*}
where $a=(1-s^2)^{-1}d^{-2} \hat{a}$, $b=(1-s^2)^{-1}d^{-2}$, $\bar{x}_n=ds$, and $s=(4\bar{b}d^2-1)/(4\bar{b}d^2+1)$. Here $\hat{a}$, $d$ and $\bar{b}$ are positive numbers given above. It is easy to verify that $a$, $b$ and $\bar{x}_n$ have full range of $\bR_+$, $\bR_+$ and $\bR$, respectively. Plugging the above form of $v$ back to the equation, we must have $f(2a^{-2}b\bm{e})=1$ and $2 a^{-1}b \bar x_n=c$.

\end{proof}

To end this section, we provide counterexamples in Remark \ref{ms_starter-optimal-remark}. In the following, we denote $r=|x|$ and do not distinguish $v(r)$ and $v(x)$ when $v$ is radially symmetric.

\begin{example}\label{reex1}
    Let $v = \alpha \log r$, $\alpha \in (0,\infty)$. By \cite[formula $(23)$]{CLL-1},
    \begin{equation*}
        \lambda(A[v]) = C(r) (1,-1,\ldots,-1)\quad \text{and}\quad
        C(r) = 2^{-1} e^{-2v}  v' (v+ 2 \log r )' >0.
    \end{equation*}
    Clearly, for any cone $\Gamma$ with $\bl\in\overline{\Gamma}$, i.e. $\mgn \leq 1$, $\lambda(A[v]) \in \overline{\Gamma}$ in $\bR^n\setminus \{0\}$. It is easy to see that 
    $\frac{\p v}{\p x_n}=0$ on $\p\hrn\setminus\{0\}$.
    However, $\liminf_{r\rightarrow 0} v(r) = -\infty$. 

    One can also construct counterexamples satisfying the boundary condition of \eqref{lowerconicalequ} for any $c<0$. Let $v(x)=\alpha\log|x|+2c x_n-10 c^2 |x|^2$ with $\alpha>0$ sufficiently small. By a similar computation as that for $w_\alpha$ in the proof of Lemma \ref{240204-1953}, for any cone $\Gamma$ with $\bl\in\Gamma$, i.e. $\mgn<1$, $\lambda(\av)\in \Gamma$ in $B_r\setminus\{0\}$ and $\frac{\p v}{\p x_n}\leq c\cdot e^v$ on $\p' B_r\setminus\{0\}$, for $r>0$ small. However, $\liminf_{x\to 0} v(x) = -\infty$.
    \end{example}

\section{Proof of Theorem \ref{nondegenerateliouville}}\label{proof-1.2}
Theorem \ref{nondegenerateliouville} can be proved in the same vein as the proof of Theorem \ref{nondegenerateliouville-critical}. We outline its proof here, focusing only on the differences.

\begin{proposition}\label{notouching-thm-av-subcritical}
    For $n\geq 2$ and $c,p\in\bR$, let $(f,\Gamma)$ satisfy \eqref{eqn-230331-0110} and \eqref{eqn-240223-0305} with $\bl\notin\overline\Gamma$.
    Suppose that $u\in C^2(\overline{B^+}\setminus\{0\})$ and $v\in C^2(\overline{B^+})$ satisfy $\lambda(\av)(\overline{B^+})\subset\Gamma$,
\begin{equation}\label{notounchingproblem-av-subcritical}
    \begin{cases}
         e^{pu} f(\lambda(A[u])) \geq e^{pv} f(\lambda(\av))~~  \text{in}~{B^+},\\
        \frac{\p u}{\p x_n}\leq c \cdot e^u~~\text{on}~\p' B^+\setminus\{0\},~~  \frac{\p v}{\p x_n}\geq c \cdot e^v~~ \text{on}~\p' B^+, 
        \end{cases}
\end{equation}
and $u>v$ in $B^+$.
Then $\liminf\limits_{x\to 0} (u-v)(x)>0$.
\end{proposition}

\begin{proof}
As in the proof of Proposition \ref{notouching-thm-av-critical}, we suppose by contradiction that $u(0)\coloneqq \liminf_{x\to 0}u(x)=v(0)$. 

Suppose that we have already reached a contradiction when $\nabla_T v(0)=0$. Now we derive a contradiction for the remaining case that $\nabla_T v(0)\neq 0$.
By Lemma \ref{vansihgrad}, there exists a M\"obius transformation $\psi$ in the form of \eqref{240201-1609} such that $\psi(0) = 0$, $\psi(\hrn)=\hrn$, and $\nabla_T v^\psi (0) = 0$.
By \eqref{notounchingproblem-av-subcritical} and its M\"obius invariance, a direct computation gives that, for some $\delta>0$,  $\lambda(A[v^\psi])(\overline{B_\delta^+})\subset\Gamma$,
$\frac{\p u^\psi}{\p x_n}\leq c \cdot e^{u^\psi}$ on $\p' B_\delta^+\setminus\{0\}$, $\frac{\p v^\psi}{\p x_n}\geq c \cdot e^{v^\psi}$ on $\p' B^+_\delta$, and
\begin{gather*}
    e^{p u^\psi} f(\lambda(A[u^\psi])) -
    e^{p v^\psi} f(\lambda(A[v^\psi]))  \\
    =  \lambda^{2p}|\cdot-\bar{x}|^{-2p}    \left( e^{p u\circ\psi} f(\lambda(A[u\circ \psi])) -
    e^{p v\circ\psi} f(\lambda(A[v\circ\psi])) \right)\geq 0 \quad \text{in}~B_\delta^+.
\end{gather*}

Since $\nabla_T v^\psi(0)=0$,  we can reach a contradiction by applying the conclusion of the former case to $u^\psi$ and $v^\psi$. 

Therefore, we only need to consider the case when $\nabla_T v(0)=0$. The remaining proof is essentially same as that of Proposition \ref{notouching-thm-av-critical} except that we should replace \eqref{240125-2006-av-1} by the following:
\begin{equation*}
\begin{cases}
\text{either}~e^{p v(x_\epsi)}f(\lambda(\av))(x_\epsi)
>
e^{p\widetilde{v}_\epsi(x_\epsi)}
f(\lambda(A[\widetilde{v_\epsi}]))(x_\epsi),\\
\text{or}~\lambda(A[\widetilde{v_\epsi}])(x_\epsi)\notin\Gamma,     
\end{cases} \quad 
\text{if}~x_\epsi\in B^+.
\end{equation*}

The same proof as that of \eqref{240125-2006-av-1} will give the above, except that one should replace the definition of $F$ there by the following:
\begin{equation*}
 F(s,p,M)\coloneqq e^{ps}f(\lambda(e^{-2s}(-M+p\otimes p-2^{-1}|p|^2 I))).
\end{equation*}
Hence, Proposition \ref{notouching-thm-av-subcritical} is proved.
\end{proof} 

\begin{proof}[Proof of Theorem \ref{nondegenerateliouville}]
\textbf{Step 1:} Same as the proof of Theorem \ref{nondegenerateliouville-critical}, we can show that for any $x\in \p \hrn$, there exists $\lambda_0(x)>0$ such that \eqref{step1} holds.

Therefore, same as before, we can define $\bar{\lambda}(x)$ for every $x\in\p\hrn$. 

\textbf{Step 2:} Prove that if $\bar\lambda(x)<\infty$ for some $x\in\p\hrn$, then $v^{x,\bar\lambda(x)}\equiv v$ in $\overline \hrn\setminus\{x\}$.
Same as before, we only need to prove \eqref{step2}.

A calculation gives, using \eqref{halfspace-equ-v} and the M\"obius invariance,
\begin{equation}\label{eqn-230528-1202-sub}
f(\lambda(A[v^{\bar{\lambda}}]))= (\frac{\bar\lambda}{|y|})^{2p} e^{-p v^{\bar\lambda}}\leq e^{-p v^{\bar\lambda}}
        ~\text{in}\,\, \overline\hrn \setminus {B_{\bar{\lambda}}},
        ~\text{and}~
        \frac{\p v^{\bar{\lambda}}}{\p x_n}=c\cdot e^{v^{\bar{\lambda}}}~\text{on}~\p\hrn\setminus B_{\bar{\lambda}}.
    \end{equation}
Note also from \eqref{halfspace-equ-v}
   \begin{equation*}
f(\lambda(A[v]))= e^{-p v}
        ~\text{in}\,\, \overline\hrn \setminus {B_{\bar{\lambda}}},
        ~\text{and}~
        \frac{\p v}{\p x_n}=c\cdot e^{v}~\text{on}~\p\hrn\setminus B_{\bar{\lambda}}.
    \end{equation*}
From the above and \eqref{eqn-230528-1202-sub}, we have
\begin{equation*}
    f(\lambda(A[v^{\bar{\lambda}}]))-f(\lambda(\av))
    -(e^{-p v^{\bar\lambda}}-e^{-pv})\leq 0\quad \text{in}~\overline\hrn\setminus B_{\bar\lambda}.   
\end{equation*}
Hence, using the strong maximum principle and the Hopf Lemma as before, we may assume \eqref{eqn-230528-1207} holds.
Moreover, an application of the Hopf Lemma and \cite[Lemma 10.1] {MR2001065} as before gives 
\eqref{eqn-230528-1212}.
Next we prove \eqref{eqn-230528-1209}.
Similar to \eqref{eqn-230528-1202-sub}, we have
     \begin{equation*}
f(\lambda(A[v^{\bar{\lambda}}]))\geq e^{-p v^{\bar\lambda}}
        ~\text{in}\,\, B^+_{\bar{\lambda}},
        ~\text{and}~
        \frac{\p v^{\bar{\lambda}}}{\p x_n}=c\cdot e^{v^{\bar{\lambda}}}~\text{on}~\p' B^{+}_{\bar{\lambda}}\setminus \{0\}.
\end{equation*}
Applying Proposition \ref{notouching-thm-av-subcritical}, we obtain \eqref{38--reformulate}, i.e. \eqref{eqn-230528-1209} holds.
Then we can complete Step 2 by the same proof as that of Theorem \ref{nondegenerateliouville-critical}.

\textbf{Step 3:} Same as before, we can prove that either $\bar{\lambda}(x)<\infty$ for all $x\in \p\hrn$ or $\bar{\lambda}(x)=\infty$ for all $x\in\p\hrn$.

By Step 3, the remainder of the proof is divided into two cases. 

\noindent
\textbf{Case 1:} $\bar{\lambda}(x)=\infty$ for all $x\in\p\hrn$.
As before, by \cite[lemma 11.3]{MR2001065}, this case leads to Alternative (ii).

\noindent
\textbf{Case 2:} $\bar{\lambda}(x)<\infty$ for all $x\in \p\hrn$. As before, we can show that \eqref{240216-1706} and \eqref{240216-1707} hold. It follows from \eqref{240216-1706}, by equation \eqref{halfspace-equ-v} and the M\"obius invariance, that we have $p=0$. By the same arguments as before, we can show that Alternative (i) holds.
\end{proof}

\section{Solutions on one variable}\label{ode-sec}
In this section, we work with function $v$ on one variable $x_n$.
We do not distinguish between $v(x)$ and $v(x_n)$, and write $v'$ for $d v / d x_n$.
A direct computation gives that 
\begin{equation}\label{240304-1443}
        \lambda(A[v]) = (\lambda_1,\lambda_2,\dots,\lambda_2),
\end{equation}
where $\lambda_1=-v'' e^{-2v}+2^{-1}(v')^2 e^{-2v} $ and $\lambda_2= -2^{-1}(v')^2 e^{-2v}\leq 0$.

\medskip

For any $\Gamma$ satisfying \eqref{eqn-230331-0110} with $\bm{e_n}\in\p\Gamma$, it must hold that $\lambda(\av)\notin\Gamma$ pointwisely for any function $v=v(x_n)$. Hence, the following discussion focuses on the case of $\bm{e_n}\notin\p\Gamma$. The following lemma is useful.
\begin{lemma}\label{xn-convexity}
    For $n\geq 2$, let $\Gamma$ satisfy \eqref{eqn-230331-0110} with $\bm{e_n}\notin\p\Gamma$, i.e. $\mgn<\infty$. Assume that a function $v$ on one variable $x_n$ satisfies $\lambda(\av)\in\overline\Gamma$ (resp. $\Gamma$) at some $x_0\in\bR^n$.
    \begin{enumerate} [label=(\alph*)]
    \item If $\bl\notin\p\Gamma$, i.e. $\mgn\neq 1$, then $\frac{\mgn-1}{2}(e^{\frac{\mgn-1}{2}v})''\leq 0$ (resp. $<0$) at $x_0$.
    \item If $\bl\in\p\Gamma$, i.e. $\mgn=1$, then $v''\leq 0$ (resp. $<0$) at $x_0$.
    \end{enumerate}
\end{lemma}
\begin{proof}
When $\bl\notin\p\Gamma$, plugging $\varphi= e^{2^{-1}{(\mgn-1)}{v}}$ into formula \eqref{240304-1443}, then $\lambda(\av)=c(a+ b \mgn ,-b\dots,-b)$, where $a=-(\mgn-1)\varphi''\varphi^{-1}$, $b=(\varphi')^2 \varphi^{-2}\geq 0$ and $c=2(\mgn-1)^{-2}e^{-2v}>0$.
Since $\lambda(\av)\in\overline\Gamma$ (resp. $\Gamma$), the form of $\lambda(\av)$ implies $a\geq 0$ (resp.$>0$), and thus part (a) follows. Part (b) follows directly from formula \eqref{240304-1443}.
\end{proof}

The above lemma implies the following:         
For $n\geq 2$, let $\Gamma$ satisfy \eqref{eqn-230331-0110} with $\bl\notin\overline\Gamma$ and $\bm{e_n}\notin\p\Gamma$, i.e. $1<\mgn<\infty$. Then there exists no function $v\in C^2(\overline\hrn)$ on one variable $x_n$ satisfying $\lambda(\av)\in\Gamma$ on $\overline\hrn$ and $v'(0)\leq 0$.

Indeed, suppose the contrary that there is a such $v$. By Lemma \ref{xn-convexity} (a), $v''<-2^{-1}{(\mgn-1)}(v')^2$ in $[0,\infty)$.
Since $v''<0$, we have $v'(\epsi)<0$ for $\epsi>0$. Integrating the above differential inequality, we obtain that $v'(x_n)\leq 1/(2^{-1}{(\mgn-1)}(x_n-\epsi)+(v'(\epsi))^{-1})$ for any $x_n\geq \epsi$. This leads to $v(t)'\to -\infty$ as $t\to T-$, for some $T>0$. A contradiction.

\subsection{Nonexistence of entire solutions}
For $n\geq 2$, let $(f,\Gamma)$ satisfy \eqref{eqn-230331-0110} and \eqref{eqn-240223-0305} with $\bl\notin\overline\Gamma$. For constants $c\in\bR$ and $p\geq 0$, we study the question when equation \eqref{halfspace-equ-v} has no solution $v$ on one variable $x_n$. 

\begin{proposition}\label{affine-ode-blowup}
    Let $n$, $c$, $p$, and $(f,\Gamma)$ satisfy the assumptions in Theorem \ref{nondegenerateliouville}. Assume that $\bm{e_n}\notin\p\Gamma$, i.e. $\mgn<\infty$, $c>0$, and \eqref{nice-ode-condition} hold. If $p\in [0,\mgn+1]$,
    then equation \eqref{halfspace-equ-v} has no solution $v\in C^2(\overline\hrn)$ on one variable $x_n$.
\end{proposition}

\begin{remark}\label{supple-remark}
    Let $n$, $c$, $p$, and $(f,\Gamma)$ satisfy the assumptions in Theorem \ref{nondegenerateliouville}. If either $\bm{e_n}\in\p\Gamma$ or $c\leq 0$ holds, then equation \eqref{halfspace-equ-v} has no solution $v\in C^2(\overline\hrn)$ on one variable $x_n$. Note here we do not assume \eqref{nice-ode-condition} nor restrictions on $p$. This follows easily from the discussions above and below Lemma \ref{xn-convexity}.
\end{remark}

\begin{remark}\label{remark-optimal-range}
  The range $p\in[0,\mgn+1]$ in Proposition \ref{affine-ode-blowup} is optimal:  For any $p>\mgn+1$, $c>0$, and any cone $\Gamma$ with $\bl\notin\overline\Gamma$ and $\bm{e_n}\notin\p\Gamma$, there exist a smooth function $f$ satisfying \eqref{eqn-240223-0305} and \eqref{nice-ode-condition}, and a smooth solution $v$ on one variable $x_n$ of \eqref{halfspace-equ-v}.
\end{remark}

Conditions \eqref{nice-ode-condition} and $p\in[0,\mgn+1]$ in Proposition \ref{affine-ode-blowup} can be weakened to more general ones. See the discussions at the end of this subsection. 
\begin{proof}[Proof of Proposition \ref{affine-ode-blowup}]
Suppose the contrary that there is a such $v$.
By the discussion below Lemma \ref{xn-convexity}, we must have
\begin{equation}\label{240304-1700}
    v'>0\quad \text{in}~[0,\infty).
\end{equation}

By Lemma \ref{xn-convexity} and $\bl\notin\overline\Gamma$, we have
\begin{equation}\label{240304-1714}
    (e^{2^{-1}{(\mgn-1)}v})''<0\quad \text{in}~[0,\infty),
\end{equation}
i.e. $-v''-2^{-1}(\mgn-1)(v')^2>0$ in $[0,\infty)$.  Integrating this, we obtain that $v'(x_n)\leq 1/(2^{-1}{(\mgn-1)} x_n+(v'(0))^{-1})$ for $x_n\geq 0$. In particular, 
\begin{equation}\label{240304-1719}
   v'(x_n)\to 0\quad \text{as}~~x_n\to\infty.
\end{equation}

By \eqref{240304-1700}, we have  $\lim_{x_n\to\infty}v(x_n)\in (-\infty,\infty]$. The remainder of the proof is divided into two cases. 

\textbf{Case 1:} $\lim_{x_n\to\infty} v(x_n)<\infty$. 

By \eqref{240304-1443} and \eqref{240304-1719}, we have
\begin{equation*}
    \lambda(\av)(x_n)= (-v''(x_n)e^{-2 v(x_n)}, 0,\dots,0)+o(1)\quad\text{as}~x_n\to \infty.
\end{equation*}
By \eqref{nice-ode-condition}, we have $0\notin \overline{f^{-1}[t_1,t_2]}$ for any $0<t_1<t_2<\infty$.
Combining the above and equation \eqref{halfspace-equ-v}, there exists some constant $\delta_0>0$ such that
\begin{equation*}
    -v''(x_n) e^{-2 v(x_n)}\geq \delta_0\quad \text{for}~x_n~ \text{large}.
\end{equation*}
By \eqref{240304-1700}, $-v''(x_n)\geq \delta_0 e^{2 v(0)}$ for $x_n$ large. Integrating this gives $v'(x_n)\leq b-\delta_0 e^{2 v(0)}x_n$ for some $b\in\bR$. A contradiction with \eqref{240304-1700} for $x_n$ large.

\textbf{Case 2:} $\lim_{x_n\to\infty} v(x_n)=\infty$. 

By \eqref{240304-1714}, $e^{2^{-1}{(\mgn-1)}v}v'$ is strictly decreasing $x_n$. This implies, by \eqref{240304-1700}, that 
\begin{equation}\label{240304-1913}
    s \coloneqq \lim_{x_n\to\infty} e^{2^{-1}{(\mgn-1)}v(x_n)}v'(x_n)\in [0,\infty).
\end{equation}
By \eqref{240304-1443} and \eqref{240304-1913}, we have
\begin{align}
&\quad \quad \lambda(\av)=  e^{-(\mgn+1) v}  (P+Q)\quad\text{with} \label{decompse-1}\\
P\coloneqq & e^{(\mgn-1)v} (-v''-2^{-1}{(\mgn-1)}(v')^2) \cdot(1,0,\dots,0)\quad\text{and} \notag \\
 &Q\coloneqq 2^{-1}{s^2}\cdot  (\mgn,-1,\dots,-1)+o(1)\quad \text{as}~x_n\to\infty. \notag
\end{align}
By equation \eqref{halfspace-equ-v}, the homogeneity of $f$ and $p\leq \mgn+1$, we have $f(\lambda(P+Q))=e^{(\mgn+1-p)v}\geq 1$ for $x_n$ large. On the other hand, $Q\to \p\Gamma$ as $x_n\to\infty$. Hence, by \eqref{eqn-230331-0110} and \eqref{nice-ode-condition}, there exists some $\delta_0>0$ such that
\begin{equation}\label{H2--inequ}
    e^{(\mgn-1)v}(-v''-2^{-1}{(\mgn-1)}(v')^2) \geq \delta_0\quad\text{for}~ x_n~\text{large}.
\end{equation}
Multiplying by $v'$, the above implies 
\begin{equation*}
    (e^{(\mgn-1)v}(v')^2 +2\delta_0 v)'\leq 0\quad \text{for}~x_n~\text{large}.
\end{equation*}
Integrating this gives $e^{(\mgn-1)v}(v')^2 +2\delta_0 v\leq a$ for some constant $a\in \bR$. Dropping the first term implies that $v(x_n)\leq (2\delta_0)^{-1}a$ for $x_n$ large. A contradiction.
Proposition \eqref{affine-ode-blowup} is proved.
\end{proof}

To end this subsection, we introduce the following more general conditions than conditions \eqref{nice-ode-condition} and $p\in[0,\mgn+1]$:
\begin{equation}\label{H1}
    \exists q\in [0,\mgn+1),~\sigma>0,~\text{s.t.}~\limsup_{t\to\infty} \big( e^{pt}\cdot \sup_{\lambda\in B_{\sigma}\cap \Gamma} f(e^{-qt}\lambda)\big)<1,\tag{H1}
\end{equation}
\begin{equation}\label{H2}
    \forall s\geq 0,~\exists \sigma_s>0,~\text{s.t.}~\limsup_{t\to\infty} \big( e^{pt}\cdot \sup_{\lambda\in B_{\sigma_s}(\lambda^s)\cap \Gamma} f(e^{-(\mgn+1)t}\lambda)\big)<1,\tag{H2}
\end{equation}
where $\lambda^{s}\coloneqq s(\mgn,-1,\dots,-1)\in\p\Gamma$.
When $f$ is homogeneous of degree $1$ and $0\notin\overline{f^{-1}(1)}$, condition \eqref{H1} is equivalent to $p\in [0,\mgn+1)$. When $f$ is homogeneous of degree $1$ and $\lambda_s\notin \overline{f^{-1}(1)}$ for any $s\geq 0$, condition \eqref{H2} is equivalent to $p\in [0,\mgn+1]$. In particular, if $f$ satisfies \eqref{nice-ode-condition}, the condition ``either \eqref{H1} or \eqref{H2} holds" is equivalent to $p\in[0,\mgn+1]$.

Let $n$, $c$, $p$, and $(f,\Gamma)$ satisfy the assumptions in Theorem \ref{nondegenerateliouville}. Assume that $\bm{e_n}\notin\p\Gamma$, $c>0$, and that either \eqref{H1} or \eqref{H2} holds. Then the conclusion of Proposition \ref{affine-ode-blowup} holds. This can be proved in the same spirit as the proof of Proposition \ref{affine-ode-blowup}. We sketch its proof as below.

We argue by contradiction as the proof of Proposition \ref{affine-ode-blowup}. Same as before, we can prove \eqref{240304-1700}--\eqref{240304-1719}, and $\lim_{x_n\to\infty}v(x_n)\in (-\infty,\infty]$. The remainder of the proof is divided into two cases as before. Case 1: $\lim_{x_n\to\infty} v(x_n)<\infty$. We can reach a contradiction by the same arguments as before since both conditions \eqref{H1} and \eqref{H2} imply that $0\notin\overline{f^{-1}(1)}$ if $p=0$, and $0\notin\overline{f^{-1}(t)}$ for any $t>0$ if $p>0$. Case 2: As before, \eqref{240304-1913} holds. When condition \eqref{H2} holds, by decomposing $\lambda(\av)$ as the form of \eqref{decompse-1}, we can still derive \eqref{H2--inequ} which leads to a contradiction. When condition \eqref{H1} holds,  rewriting formula \eqref{240304-1443} as below:
\begin{equation*}
    \lambda(\av)= e^{-q v} \{(-v'' e^{(q-2)v},0,\dots,0) + 2^{-1}(e^{2^{-1}{(q-2)}v}v')^2 (1,-1,\dots,-1) \}.
\end{equation*}
By \eqref{240304-1913} and $q<\mgn+1$, we have $\lim_{x_n\to\infty} e^{2^{-1}{(q-2)}v}v'=0$. Therefore,
\begin{equation*}
    \lambda(\av)= e^{-q v} \{(-v'' e^{(q-2)v},0,\dots,0) + o(1) \}\quad\text{as}~x_n\to\infty.
\end{equation*}
By equation \eqref{halfspace-equ-v} and condition \eqref{H1}, there exists some $\delta_0>0$ such that
\begin{equation}\label{240304-2108}
    -v''(x_n) e^{(q-2)v(x_n)}\geq \delta_0\quad \text{for}~x_n~\text{large}.
\end{equation}
When $q=2$, integrating the above implies $v'(x_n)\leq -\delta_0 x_n + a$ for some $a\in\bR$. A contradiction with \eqref{240304-1700} for $x_n$ large. When $q\neq 2$, multiplying both sides of \eqref{240304-2108} by $v'$ leads to $[2^{-1}(v')^2 +{\delta_0}{(2-q)^{-1}}e^{(2-q)v}]'\leq 0$ for $x_n$ large. Integrating this gives $(v'(x_n))^2\leq b-{2\delta_0}{(2-q)^{-1}} e^{(2-q)v(x_n)}$ for some $b\in\bR$. A contradiction for $x_n$ large.

\subsection{Counterexamples in Remark \ref{remark-optimal-range}}

Recall \eqref{240304-1443}, when $v=v(x_n)$, $\lambda(A[v])$ takes the form $(\lambda_1, \lambda_2, \ldots, \lambda_2)$. We solve for $v$ such that $\lambda_1 + \mgn \lambda_2 = e^{-pv}$, which is equivalent to
\begin{equation} \label{eqn-240918-1018}
    v'' + \frac{\mgn - 1}{2} (v')^2 + e^{(2-p)v} = 0
\end{equation}

\begin{lemma}\label{lem-240919-0615}
For $v_0,w_0\in\bR$, $\mgn > 1$, and $p > \mgn + 1$, there exists a unique smooth solution $v$ of \eqref{eqn-240918-1018} with $v(0) = v_0$, $v'(0) = w_0$ in $[0,\infty)$ if and only if $w_0 \geq \sqrt{2} (p-\mgn-1)^{-1/2} e^{-(p-2)v_0/2}$.
\end{lemma}

Lemma \ref{lem-240919-0615} implies that, when $\mgn > 1, p > \mgn + 1$, and $c>0$, we can solve \eqref{eqn-240918-1018} with initial conditions $v(0) = v_0$ and $v'(0) = ce^{v_0}$ with $v_0 \geq -p^{-1}(2\log c + \log (p - \mgn - 1) - \log 2)$ for a solution $v=v(x_n)$, $x_n \in [0, \infty)$. Such $v$ satisfies \eqref{halfspace-equ-v} with an $f$ constructed similarly to that in \cite[Example~5.1]{CLL-1}.
While the overall construction mirrors the example, specific adaptations are necessary for different $\mgn$ values. For $\mgn \in (1, n-1)$, the construction remains unchanged; for $\mgn = n-1$, we take $f = \sigma_1$; and for $\mgn > n-1$, we begin with $f^{(0)} (\lambda) := \min_k \{ \lambda_k + \mgn (n-1)^{-1} \sum_{j \neq k} \lambda_j \}$, with subsequent steps following the established procedure.

\begin{proof}[Proof of Lemma \ref{lem-240919-0615}]
    For convenience, denote $\theta = (\mgn -1)/2 (>0)$ and $q = (p-2)/2 (> \theta)$. From $\mgn > 1$ and $p > \mgn + 1$, there holds $q > \theta > 0$. 
    Let $(T_{-}, T_+)$ be the maximal existence interval of solution $v$ of \eqref{eqn-240918-1018} with $v(0) = v_0$ and $v'(0) = w_0$. Below, we only discuss $T_+$. 
    
    Let $\varphi = e^v$ and $w = v'$. Then $\varphi > 0$ and $w$ satisfy
    \begin{equation} \label{eqn-240919-1215}
        \begin{cases}
            \varphi' = \varphi w,\\
            w' = -\theta w^2 - \varphi^{-2q},\\
            \varphi(0) = \varphi_0, \,\, w(0) = w_0.
        \end{cases}
    \end{equation}
    where $\varphi_0 = e^{v_0} > 0$. 
    Noting that $q > \theta$, a calculation shows that \eqref{eqn-240919-1215} has a first integral
    \begin{equation} \label{eqn-240919-0535}
        I_{\theta, q} (\varphi, w) = \varphi^{2\theta}w^2 - \frac{1}{q - \theta}\varphi^{-2(q - \theta)},
    \end{equation}
    i.e.
    \begin{equation} \label{eqn-240919-0312}
        I_{\theta,q} (\varphi, w) = I_{\theta,q} (\varphi_0, w_0) (:= I_0) \quad \text{on}\,\,(T_-, T_+).
    \end{equation}
    Thus we have, on $(T_-, T_+)$, that
    \begin{equation} \label{eqn-240919-0323}
        0 \leq w^2 
        = 
        F(\varphi) 
        \equiv 
        F_{\theta,q,\varphi_0, w_0} (\varphi)
        :=
        I_0 \varphi^{-2\theta} + \frac{1}{q - \theta} \varphi^{- 2q}.
    \end{equation}

    \textbf{Case 1}: $w_0 \leq 0$. 

    In this case, we prove $T_+ < \infty$. From $w' < 0$, we have $w<0$ on $(0, T_+)$. Hence, from $\varphi ' = \varphi w$ and \eqref{eqn-240919-0323},
    \begin{equation*}
        \varphi ' = - \varphi \sqrt{F(\varphi)} < 0 \,\, on\,\, (0, T_+).
    \end{equation*}
    Let $a = \lim_{x_n \rightarrow (T_+)-} \varphi(x_n) (\in [0, \varphi_0))$. Integrating the above, we obtain
    \begin{equation*}
        T_+  = \int_a^{\varphi_0} \frac{ds}{s \sqrt{F(s)}} = \int_a^{\varphi_0} \frac{ds}{s^{1-q} \sqrt{I_0 s^{2(q-\theta)} + \frac{1}{q-\theta}}}
    \end{equation*}
    Since $[a,\varphi_0]$ is a finite interval and all potential zeros of $s^{1-q} \sqrt{I_0 s^{2(q-\theta)} + \frac{1}{q-\theta}}$ are at most of order $1-q (< 1)$ or $1/2$, there must hold $T_+ < \infty$.

    \textbf{Case 2}: $w_0 > 0$.

    We prove $T_+ < \infty$ when $I_0 <0$ and $T_+ = \infty$ when $I_0 \geq 0$. 

    When $I_0 \geq 0$, from \eqref{eqn-240919-0535} and \eqref{eqn-240919-0312}, we have $w > 0$ in $(T_-, T_+)$. Hence, from $\varphi ' = \varphi w$ and \eqref{eqn-240919-0323},
    \begin{equation*} 
        \varphi ' = \varphi \sqrt{F(\varphi)} > 0 \,\, \text{on}\,\, (0, T_+).
    \end{equation*}
    Let $a = \lim_{x_n \rightarrow (T_+)-} \varphi(x_n) (\in (\varphi_0, \infty])$. Integrating the above, we obtain
    \begin{equation*}
        T_+ = \int_{\varphi_0}^a \frac{ds}{s \sqrt{F(s)}} = \int_{\varphi_0}^a \frac{ds}{s^{1-\theta} \sqrt{I_0 + \frac{1}{q-\theta} s^{-2 (q-\theta)}}}.
    \end{equation*}
    There must hold $T_+ = \infty$ since otherwise one can continue the solution by solving the ODE system in \eqref{eqn-240919-1215} with initial conditions $\varphi(T_+) = a$ and $w(T_+) = \sqrt{F(a)}$, contradicting with the assumption that $(T_-, T_+)$ is the maximal existence interval.
    
    When $I_0 < 0$, from Case 1, it suffices to show $w(x_n) \leq 0$ for some $x_n \in (0, T_+)$. Suppose the contrary, $w > 0$ in $(0, T_+)$. Since from \eqref{eqn-240919-1215}, $w' <0$, we have $0 < w < w_0$ in $(0, T_+)$. From $\varphi' = \varphi w > 0$, $a:= \lim_{x_n \rightarrow (T_+)-} \varphi(x_n) \in (\varphi_0 ,\infty]$. From \eqref{eqn-240919-0535}, \eqref{eqn-240919-0312}, and $I_0 < 0$, there must hold $a < \infty$. These imply that $\sup_{(0,T_+)} (|v| + |v'|) < \infty$, and thus $T_+ = \infty$. However, back to the $w$-equation of \eqref{eqn-240919-1215}, we have $w' \leq - \varphi^{-2q} \leq -2a^{-2q}$. It is impossible that $w> 0$ on $(0, T_+) = (0, \infty)$. A contradiction.

    Noting that the inequality $w_0 \geq \sqrt{2} (p-\mgn-1)^{-1/2} e^{-(p-2)v_0/2}$ can be written as $w_0 > 0$ and $I_0 \geq 0$, the lemma is proved.
\end{proof}

\subsection{Counterexamples in Remark \ref{nondegoptimalremark}}\label{counterdetail}
 
\begin{example} \label{exp-230622-1021}
    Let $n\geq 2$, for any $c\in\bR$ and any $s \in (0,1]$, there exist a cone $\Gamma$ satisfying \eqref{eqn-230331-0110} with $\mgn = s$, and a symmetric homogeneous of degree $1$ function $f \in C^0(\overline{\Gamma}) \cap C^\infty(\Gamma)$ satisfying $f|_{\p \Gamma} = 0$ and $C^{-1} \leq \p_{\lambda_i} f\leq C$ for all $i$ in $\Gamma$, and a function $v \in C^\infty (\overline\hrn)$ satisfying \eqref{halfspace-equ-v-critical}, but $v$ is not of the form \eqref{half-space-bubble}.
\end{example}

 For any $s\in (0,1] $, $v_0\in\bR$ and $c\in\bR$, by \cite[Lemma 5.1]{CLL-1}, there exists a unique smooth function $v=v(x_n)$ satisfying
\begin{equation*}
\lambda_1+s\lambda_2=1~ \text{on}~ [0,\infty)~ \text{with} ~v(0)=v_0~ \text{and} ~v'(0)=ce^{v_0},
\end{equation*}
where $\lambda_1$ and $\lambda_2$ are those in \eqref{240304-1443}.
In particular, ${\p_{x_n} v}=c\cdot e^v$ on $\p\hrn$.

By the same constructions as that of \cite[Section 5.2]{CLL-1}, we can find desired $(f,\Gamma)$'s in Example \ref{exp-230622-1021} matching the above chosen solution $v$, and also desired $f$'s in Remark \ref{nondegoptimalremark} matching the above $v$. 

\appendix
\section{Equivalent theorems in terms of Ricci tensor}\label{riccisection}
For reader's convenience, we reformulate Theorem \ref{nondegenerateliouville-critical} and \ref{nondegenerateliouville} in terms of Ricci tensor. Ricci tensor and Schouten tensor are related by a linear transformation $T$ as explained below.

Recall that on a Riemannian manifold $(M^n,g)$ of dimension $n\geq 3$, we have 
\begin{equation*}
    Ric_g=(n-2)A_g + (2(n-1))^{-1}{R_g}\cdot g.
\end{equation*}
Clearly, $\lambda(Ric_g)\equiv T\lambda(A_g)$, where $T=(n-2)I+\bm{e}\otimes\bm{e}$. As $(f,\Gamma)$ denotes pairs for eigenvalues of Schouten tensor, we use $(\hat{f},\hat{\Gamma})$ to denote pairs for eigenvalues of Ricci tensor. Let
\begin{equation*}
    f(\lambda)\coloneqq \hat{f}(T\lambda),\quad \lambda\in\Gamma\coloneqq T^{-1}\hat{\Gamma}.
\end{equation*}
Then we have $f(\lambda(A_g))\equiv \hat{f}(\lambda(Ric_g))$, and $\lambda(A_g)\in\Gamma$ if and only if $\lambda(Ric_g)\in\hat\Gamma$.

\medskip

Let $(\hat{f},\hat\Gamma)$ satisfy the following conditions:
 \begin{equation} \label{eqn-230331-0110-hat}
   \begin{cases}
    \hat\Gamma \subsetneqq \bR^n \,\, \text{is a non-empty open symmetric cone with vertex at the origin},\\
    \hat\Gamma+T\Gamma_n\subset\hat\Gamma,
    \end{cases}
\end{equation}
\begin{equation} \label{eqn-240223-0305-ricci}
\begin{cases}
\hat{f} \in C^{0,1}_{loc}(\hat\Gamma)\,\,\text{is a symmetric function satisfying} ~~
    T(\nabla \hat{f})\in c(K)\bm{e}+\Gamma_n\\
    \text{a.e. $K$, $c(K)>0$, for any compact subset} ~K~\text{of}~\hat\Gamma,
\end{cases}
\end{equation}
and
\begin{equation}\label{natural-ass-ricci}
    0\notin \overline{\hat{f}^{-1}(1)}.
\end{equation}
It is easy to check that $T\Gamma_n$ is the interior of the minimal convex symmetric cone with vertex at the origin containing $(n-1,1,\dots,1)$. 
Note that condition \eqref{eqn-240223-0305-ricci} allows ${\p_{\lambda_i} \hat{f}}<0$ for some $i$.

For constant $c\in\bR$, consider the equation
\begin{equation}\label{nequation-ricci-re-hat}
\begin{cases}
    \hat{f}(\lambda(Ric_{\bar{g}_u}))= 1\quad \text{on}~\bR^n,\\
    h_{\bar{g}_u}=c\quad \text{on}~\p\hrn,    \end{cases}
\end{equation}
where
$\bar{g}_u= u^{\frac{4}{n-2}}\bar{g}$, $\bar{g}=|dx|^2$ is the flat metric, $u$ is a positive function on $\bR^n$, 
$\lambda(Ric_{g})$ denotes eigenvalues of $Ric_{g}$ with respect to $g$, and $h_{\bar{g}_u}$ is the mean curvature.

Denote
\begin{equation*}
    \hat{\lambda}^*\coloneqq (0,-1,\dots,-1).
\end{equation*}
The following rigidity theorem is equivalent to Theorem \ref{nondegenerateliouville-critical}.

\begin{theorem}
    For $n\geq 3$ and $c\in\bR$, let $(\hat{f},\hat\Gamma)$ satisfy \eqref{eqn-230331-0110-hat}--\eqref{natural-ass-ricci}, and $\hat{\lambda}^* \notin \overline{\hat\Gamma}$. Assume that $u\in C^2(\overline\hrn)$ is a positive solution of \eqref{nequation-ricci-re-hat}. Then $u\equiv e^{\frac{n-2}{2}v}$, where $v$ is of the form \eqref{half-space-bubble} with $a,b>0$ and $\bar x=(\bar x', \bar x_n)\in\bR^n$ satisfy $\hat{f}(4(n-1)a^{-2}b\bm{e})=1$ and $2 a^{-1}b \bar x_n=c$ with $\bm{e}=(1,1,\dots,1)$.
\end{theorem}
The above result fails when $\hat{\lambda}^*\in\overline{\hat\Gamma}$; see counterexamples in Remark \ref{nondegoptimalremark}.

\smallskip

For constants $c\in\bR$ and $q\geq 0$, consider the following more general equation, including \eqref{nequation-ricci-re-hat} and its subcritical cases, 
\begin{equation}\label{halfspace-equ-v-hat}
       \begin{cases}
        \hat{f}(\lambda(Ric_{\bar{g}_u}))=u^{-q}\quad \text{in}~\overline\hrn,\\
        h_{\bar{g}_u}=c\quad \text{on}~\p \hrn.
        \end{cases}
\end{equation}
Theorem \ref{nondegenerateliouville} is reformulated as below.
\begin{theorem}
    For $n\geq 3$, $q\geq 0$ and $c\in\bR$, let $(\hat{f},\hat\Gamma)$ satisfy \eqref{eqn-230331-0110-hat} and \eqref{eqn-240223-0305-ricci} with $\hat{\lambda}^*\notin \overline{\hat\Gamma}$. Assume that $u\in C^2(\overline\hrn)$ is a positive solution of \eqref{halfspace-equ-v-hat}. Then one of the following holds:
\begin{enumerate}[label=(\roman*)]
    \item It holds that $q=0$ and $u\equiv e^{\frac{n-2}{2}v}$, where $v$ is of the form \eqref{half-space-bubble},
    where $a,b>0$ and $\bar x=(\bar x', \bar x_n)\in\bR^n$ satisfy $\hat{f}(4(n-1)a^{-2}b\bm{e})=1$ and $2 a^{-1}b \bar x_n=c$ with $\bm{e}=(1,1,\dots,1)$.
     \item Solution $u$ depends only on $x_n$, i.e. $u(x',x_n)\equiv u(0',x_n)$.
      \end{enumerate}
\end{theorem}
See the discussions below Theorem \ref{nondegenerateliouville} about solutions on one variable $x_n$.

\section*{Acknowledgement}
B.Z. Chu and Y.Y. Li are partially supported by NSF Grant DMS-2247410.

\end{document}